\documentclass[11pt]{amsart}

\usepackage{a4wide}
\usepackage{amssymb}
\usepackage{mathrsfs}
\usepackage{enumerate}
\usepackage{esint}
\usepackage{titletoc}
\usepackage[colorlinks=true, urlcolor=blue, linkcolor=blue, citecolor=black]{hyperref}
\usepackage[nameinlink]{cleveref}

\theoremstyle{plain}
\newtheorem{theorem}{Theorem}[section]
\newtheorem{lemma}[theorem]{Lemma}
\newtheorem{corollary}[theorem]{Corollary}
\newtheorem{proposition}[theorem]{Proposition}

\theoremstyle{definition}
\newtheorem{definition}[theorem]{Definition}

\numberwithin{equation}{section}

\makeatletter
\@namedef{subjclassname@2020}{\textup{2020} Mathematics Subject Classification}
\makeatother

\title[The fractional $p$-Laplacian on hyperbolic spaces]{The fractional $p$-Laplacian on hyperbolic spaces}

\author{Jongmyeong Kim}
\address{Institute of Mathematics, Academia Sinica, Taipei 106319, Taiwan}
\email{jmkim@gate.sinica.edu.tw}

\author{Minhyun Kim}
\address{Department of Mathematics \& Research Institute for Natural Sciences, Hanyang University, 04763 Seoul, Republic of Korea}
\email{minhyun@hanyang.ac.kr}

\author{Ki-Ahm Lee}
\address{Department of Mathematical Sciences \& Research Institute of Mathematics, Seoul National University, 08826 Seoul, Republic of Korea}
\email{kiahm@snu.ac.kr}

\thanks{The research of Jongmyeong Kim is supported by the National Research Foundation of Korea (NRF) grant funded by the Korea government (MSIP): NRF-2016K2A9A2A13003815. Minhyun Kim gratefully acknowledges financial support by the German Research Foundation (GRK 2235 - 282638148) and the National Research Foundation of Korea (RS-2023-00252297). The research of Ki-Ahm Lee is supported by the Ministry of Education of the Republic of Korea and the National Research Foundation of Korea (RS-2025-00515707).
}

\begin{document}

\begin{abstract}
We present three equivalent definitions of the fractional $p$-Laplacian $(-\Delta_{\mathbb{H}^{n}})^{s}_{p}$, $0<s<1$, $p>1$, with normalizing constants, on hyperbolic spaces. The explicit values of the constants enable us to study the convergence of the fractional $p$-Laplacian to the $p$-Laplacian as $s \to 1^{-}$.
\end{abstract}

\maketitle


\section{Introduction} \label{sec:introduction}


Operators of fractional-order have been extensively studied not only in Euclidean spaces \cite{DNPV12}, but also on Riemannian manifolds \cite{AOCM18,BGS15,CG11,DNS19,Gon18,GZ03,GZZ18}, metric measure spaces \cite{CKK+22,CK08,FLW14,Gri03,GHL14}, discrete models \cite{CRS+18}, Lie groups \cite{CS17,CRT01,FF15,FMP+18}, and Wiener spaces \cite{CS17}. In Euclidean spaces, several equivalent definitions of the fractional Laplacian exist \cite{Kwa17} due to the simple structure of the spaces. In contrast to the case of Euclidean spaces, not all definitions are equivalent in more general spaces. For instance, one can study a regional type operator \cite{GZZ18} or a spectral type operator \cite{ST10} on Riemannian manifolds. Moreover, some definitions, such as those relying on the Fourier transform, do not even work on general Riemannian manifolds and metric measure spaces. Nonetheless, for specific Riemannian manifolds such as hyperbolic spaces and spheres, several representations of the fractional Laplacian have been established \cite{BGS15,DNS19} by means of the rich structure of these spaces.

The aim of this paper is two-fold. First, we extend representation formulas in \cite{BGS15} to the nonlinear regime on hyperbolic spaces. Specifically, we define the fractional $p$-Laplacian $(-\Delta_{\mathbb{H}^{n}})^{s}_{p}$ for $n \in \mathbb{N}$, $0 < s < 1$, and $p > 1$ and provide two additional equivalent definitions via the heat semigroup and the Caffarelli--Silvestre extension. Note that a definition based on the Fourier transform is not available due to the nonlinearity of the operator. Second, we investigate the pointwise convergence of $(-\Delta_{\mathbb{H}^{n}})^{s}_{p}u(x)$ as $s \to 1^{-}$. For this purpose, the values of the normalizing constants in these definitions are provided explicitly.

Let us define the fractional $p$-Laplacian on hyperbolic spaces as the pointwise integral representation with singular kernels. Note that hyperbolic geometry is distinguished from Euclidean geometry only when $n \geq 2$.

\begin{definition} \label{def:kernel}
Let $n \geq 2$, $0<s<1$, and $p > 1$. The \emph{fractional $p$-Laplacian} on $\mathbb{H}^{n}$ is defined by
\begin{equation} \label{eq:pw-representation}
(-\Delta_{\mathbb{H}^{n}})^{s}_{p} u(x) = c_{n, s, p} \, \mathrm{P.V.} \int_{\mathbb{H}^{n}} |u(x)-u(\xi)|^{p-2} (u(x)-u(\xi)) \mathcal{K}_{n,s,p}(d(x,\xi)) \,\mathrm{d}\xi
\end{equation}
with the kernel $\mathcal{K}_{n,s,p}$ given by
\begin{equation*}
\mathcal{K}_{n,s,p}(\rho) = C_{1} \left( \frac{-\partial_\rho}{\sinh \rho} \right)^{\frac{n-1}{2}} \left( \rho^{-\frac{1+sp}{2}} K_{\frac{1+sp}{2}} \left(\frac{n-1}{2}\rho \right) \right)
\end{equation*}
when $n \geq 3$ is odd and
\begin{equation*}
\mathcal{K}_{n,s,p}(\rho) = C_{1} \int_{\rho}^{\infty} \frac{\sinh r}{\sqrt{\pi}\sqrt{\cosh r - \cosh \rho}} \left( \frac{-\partial_r}{\sinh r} \right)^{\frac{n}{2}} \left( r^{-\frac{1+sp}{2}} K_{\frac{1+sp}{2}}\left( \frac{n-1}{2}r \right) \right) \mathrm{d}r
\end{equation*}
when $n \geq 2$ is even, where
\begin{equation}\label{eq-cnsp}
c_{n, s, p} = \frac{p}{2} \frac{\sqrt{\pi}/2}{\Gamma(\frac{p+1}{2})} \frac{2^{2s} \Gamma(\frac{n+sp}{2})}{\pi^{\frac{n}{2}} |\Gamma(-s)|}, \quad C_{1} = \frac{1}{2^{\frac{n-2+sp}{2}} \Gamma(\frac{n+sp}{2})} \left( \frac{n-1}{2} \right)^{\frac{1+sp}{2}},
\end{equation}
and $K_{\nu}$ is the modified Bessel function of the second kind.
\end{definition}

For the linear case $p=2$, the pointwise integral representation with singular kernel is provided in \cite[Theorem 2.4 and 2.5]{BGS15} without constants. The main tool in \cite{BGS15} is the Fourier transform on hyperbolic spaces, but it is not available in the nonlinear setting. Our definition is motivated by the nonlinear extension of Bochner's definition \cite{Boc49}, which will be given in \Cref{thm-Bochner} below.

The normalizing constant $C_{1}$ in \Cref{def:kernel} is carefully chosen so that the pointwise convergence $\lim_{s \nearrow 1} (-\Delta_{\mathbb{H}^{n}})^{s}_{p} u(x) = (-\Delta_{\mathbb{H}^{n}})_{p} u(x)$ holds (see \Cref{thm:convergence}). We emphasize that the normalizing constant plays a crucial role in some contexts. For instance, it is used in the robust regularity theory (see \cite{CS09,KKL21}).

The singular kernel $\rho^{-n-sp}$ of the fractional $p$-Laplacian on Euclidean space $\mathbb{R}^{n}$ is homogeneous of degree $-n-sp$. This is a natural consequence of the invariance of the scale of the operator. However, such homogeneity cannot be expected in the hyperbolic setting because hyperbolic geometry comes into play. In fact, \Cref{prop-asymptotic} below exhibits the behavior of the kernel $\mathcal{K}_{n, s, p}$.

Here and in what follows, we write $A(\rho) \sim B(\rho)$ as $\rho\to 0^+$ (resp.\ $\rho=\infty$) to mean that there exist constants $c, C>0$ and $\rho_0>0$ such that $cB(\rho) \leq A(\rho) \leq CB(\rho)$ for all $\rho \in (0, \rho_0)$ (resp.\ $\rho \in (\rho_0, \infty)$).

\begin{proposition}\label{prop-asymptotic}
There exist constants $c, C>0$ such that
\begin{equation*}
c\rho^{-\frac{1+sp}{2}} (\sinh\rho)^{-\frac{n-1}{2}} K_{\frac{n+sp}{2}}\left(\frac{n-1}{2}\rho\right) \leq \mathcal{K}_{n,s,p}(\rho) \leq C\rho^{-\frac{1+sp}{2}} (\sinh\rho)^{-\frac{n-1}{2}} K_{\frac{n+sp}{2}}\left(\frac{n-1}{2}\rho\right)
\end{equation*}
for all $\rho > 0$. In particular,
\begin{equation*}
\mathcal{K}_{n,s,p}(\rho) \sim \rho^{-n-sp}
\end{equation*}
as $\rho \to 0^+$ and
\begin{equation*}
\mathcal{K}_{n,s,p}(\rho) \sim \rho^{-1-\frac{sp}{2}}e^{-(n-1)\rho}
\end{equation*}
as $\rho \to \infty$.
\end{proposition}

We remark that the well-definedness of $(-\Delta_{\mathbb{H}^n})_p^su(x)$ for $u \in C^2_b(\mathbb{H}^n)$ and $x \in \mathbb{H}^n$ (assuming also that $\nabla u(x) \neq 0$ if $p \in (1, \frac{2}{2-s}]$), where $C^{2}_{b}(\mathbb{H}^{n})$ denotes the space of bounded $C^{2}$-functions on $\mathbb{H}^{n}$, can be checked by using \Cref{prop-asymptotic}. Indeed, the proof of \cite[Lemma~3.6]{KKL19} works on hyperbolic spaces in the exact same way, once we choose sufficiently small $\varepsilon>0$ so that $\mathcal{K}_{n, s, p}(d(x, \xi)) \sim d(x, \xi)^{-n-sp}$ in the ball $B_\varepsilon(x)=\{\xi \in \mathbb{H}^n: d(x, \xi) < \varepsilon\}$ with the help of \Cref{prop-asymptotic}. This proves the finiteness of the integral in \eqref{eq:pw-representation} over $B_\varepsilon(x)$. On the other hand, by using \Cref{prop-asymptotic} again, we find a large $R>0$ so that $\mathcal{K}_{n, s, p}(d(x, \xi)) \sim d(x, \xi)^{-1-\frac{sp}{2}}e^{-(n-1)d(x, \xi)}$ outside $B_R(x)$. This proves that the integral in \eqref{eq:pw-representation} over $\mathbb{H}^n \setminus B_R(x)$ is estimated by
\begin{equation*}
\max\{2^{p-2}, 1\} \|u\|_{L^\infty(\mathbb{H}^n)}^{p-1} \int_{\mathbb{S}^{n-1}} \int_R^\infty \rho^{-1-\frac{sp}{2}}e^{-(n-1)\rho} \sinh^{n-1} \rho \,\mathrm{d}\rho \,\mathrm{d}\omega,
\end{equation*}
which is finite. The finiteness of the remaining integral over $B_R(x) \setminus B_\varepsilon(x)$ is obvious since the integrand is bounded in a bounded region.

As the first main result of this paper, we provide an equivalent representation of the fractional $p$-Laplacian on hyperbolic spaces via the heat semigroup. To this end, let $\lbrace e^{t\Delta_{\mathbb{H}^{n}}} \rbrace_{t \geq 0}$ denote the heat semigroup generated by the Laplacian $\Delta_{\mathbb{H}^{n}}$ on $\mathbb{H}^n$. That is, for a given function $f: \mathbb{H}^n \to \mathbb{R}$ we denote by $e^{t\Delta_{\mathbb{H}^{n}}}[f](x)$ the solution $w(x,t)$ of the Cauchy problem
\begin{equation} \label{eq:Cauchy}
\begin{cases}
\partial_{t} w(x,t) - \Delta_{\mathbb{H}^{n}} w(x,t) = 0, &x \in \mathbb{H}^{n}, t > 0, \\
w(x,0) = f(x), &x \in \mathbb{H}^{n}.
\end{cases}
\end{equation}

\begin{theorem}\label{thm-Bochner}
Let $n \in \mathbb{N}$, $0<s<1$, and $p > 1$. Let $u \in C^{2}_{b}(\mathbb{H}^{n})$ and $x \in \mathbb{H}^{n}$. If $p \in (1,\frac{2}{2-s}]$, assume in addition that $\nabla u(x) \neq 0$. The {\it fractional $p$-Laplacian on $\mathbb{H}^n$} is defined by
\begin{equation}\label{eq-Bochner}
(-\Delta_{\mathbb{H}^{n}})^{s}_{p} u(x) = C_{2} \int_{0}^{\infty} e^{t\Delta_{\mathbb{H}^{n}}}[ \Phi_{p}(u(x)-u(\cdot)) ](x) \frac{\mathrm{d}t}{t^{1+\frac{sp}{2}}},
\end{equation}
where
\begin{equation*}
C_{2} = \frac{p}{2} \frac{\sqrt{\pi}/2}{\Gamma(\frac{p+1}{2})} \frac{2^{s(2-p)}}{|\Gamma(-s)|}
\end{equation*}
and $\Phi_{p}(r) = |r|^{p-2}r$.
\end{theorem}

Note that the formula \eqref{eq-Bochner} on manifolds including hyperbolic spaces is given in \cite[Section~3.1]{BGS15} for the linear case $p=2$, and that it was proposed as the definition of the fractional $p$-Laplacian on Riemannian manifolds in \cite[Section~8.2]{dTGCV21}.

We now turn to another representation of the fractional $p$-Laplacian on $\mathbb{H}^{n}$. We recall that the fractional Laplacian on $\mathbb{R}^{n}$ can be realized as a Dirichlet-to-Neumann map via the Caffarelli--Silvestre extension \cite{CS07}. Later, the article \cite{ST10} relates the heat semigroup to this extension. Moreover, this relation is extended to the nonlinear framework \cite{dTGCV21} in $\mathbb{R}^{n}$. In this paper, we further investigate this relation on hyperbolic spaces. Let us consider the extension problem
\begin{equation} \label{eq:ext-prob}
\begin{cases}
\displaystyle\Delta_{x} U(x,y) + \frac{1-sp}{y} U_{y}(x,y) + U_{yy}(x,y) = 0, &x \in \mathbb{H}^{n}, y > 0, \\
U(x,0) = f(x), &x \in \mathbb{H}^{n}.
\end{cases}
\end{equation}
We will show in \Cref{lem:Poisson} that the function $U$ defined by
\begin{equation*}
U(x,y) = \int_{\mathbb{H}^{n}} P(d(x,\xi), y) f(\xi) \,\mathrm{d}\xi,
\end{equation*}
where $P$ is the Poisson kernel given in the same lemma, is a solution of \eqref{eq:ext-prob}. We define an extension operator $E_{s, p}$ by $E_{s, p}[f] := U$. The following theorem is our next main result.

\begin{theorem} \label{thm:extension}
Let $n \in \mathbb{N}$, $0<s<1$, and $p > 1$. Let $u \in C^{2}_{b}(\mathbb{H}^{n})$ and $x \in \mathbb{H}^{n}$. If $p \in (1,\frac{2}{2-s}]$, assume $\nabla u(x) \neq 0$ additionally. Then
\begin{equation*}
\begin{split}
(-\Delta_{\mathbb{H}^{n}})^{s}_{p} u(x)
&= C_{3} \lim_{y \searrow 0} \frac{E_{s, p}[\Phi_{p}(u(x)-u(\cdot))](x, y)}{y^{sp}} \\
&= \frac{C_{3}}{sp} \lim_{y \searrow 0} y^{1-sp} \partial_{y} \Big( E_{s, p}[\Phi_{p}(u(x)-u(\cdot))] \Big) (x, y),
\end{split}
\end{equation*}
where
\begin{equation*}
C_{3} = \frac{p}{2} \frac{\sqrt{\pi}/2}{\Gamma(\frac{p+1}{2})} \frac{2^{2s} \Gamma(\frac{sp}{2})}{|\Gamma(-s)|}.
\end{equation*}
\end{theorem}

The last result is the pointwise convergence of the fractional $p$-Laplacian on $\mathbb{H}^{n}$ as $s \to 1^{-}$. As one can expect, the fractional $p$-Laplacian converges to the $p$-Laplacian as a limit. Recall that the $p$-Laplacian on $\mathbb{H}^{n}$ is defined by $(-\Delta_{\mathbb{H}^{n}})_{p} u(x) = -\mathrm{div}(|\nabla u(x)|^{p-2} \nabla u(x))$.

\begin{theorem} \label{thm:convergence}
Let $n \in \mathbb{N}$, $p \geq 2$, and $u \in C^{2}_{b}(\mathbb{H}^{n})$. For $x \in \mathbb{H}^{n}$ such that $\nabla u(x) \neq 0$,
\begin{equation*}
\lim_{s \nearrow 1} (-\Delta_{\mathbb{H}^{n}})^{s}_{p} u(x) = (-\Delta_{\mathbb{H}^{n}})_{p} u(x).
\end{equation*}
\end{theorem}

The pointwise convergence of the fractional $p$-Laplacian on Euclidean spaces is well known \cite{BS22,DNPV12,IN10}. Recall that the proof uses Taylor's theorem and the following computations:
\begin{equation} \label{eq:integrals}
\begin{split}
\int_{\mathbb{S}^{n-1}} \int_{R}^{\infty} K(\rho)\rho^{n-1} \,\mathrm{d}\rho \,\mathrm{d}\omega &= \frac{|\mathbb{S}^{n-1}|}{sp} R^{-sp}, \\
\int_{\mathbb{S}^{n-1}} \int_{0}^{R} K(\rho)\rho^{p+n-1} \,\mathrm{d}\rho \,\mathrm{d}\omega &= \frac{|\mathbb{S}^{n-1}|}{p(1-s)} R^{p(1-s)}, \\
\int_{\mathbb{S}^{n-1}} \int_{0}^{R} K(\rho)\rho^{\beta+p+n-1} \,\mathrm{d}\rho \,\mathrm{d}\omega &= \frac{|\mathbb{S}^{n-1}|}{\beta +p(1-s)} R^{\beta+p(1-s)},
\end{split}
\end{equation}
where $\beta > 0$ and $K(\rho) = \rho^{-n-sp}$ is the kernel for the fractional $p$-Laplacian on $\mathbb{R}^{n}$. However, in our framework we need the integrals in \eqref{eq:integrals} with the kernel $K$ and the volume element $\rho^{n-1} \,\mathrm{d}\rho \,\mathrm{d}\omega$ replaced by $\mathcal{K}_{n,s,p}$ and $\sinh^{n-1} \rho \,\mathrm{d}\rho \,\mathrm{d}\omega$, respectively. These integrals do not seem to be of a form that is easily computed. Instead, we compute the limits of these integrals as $s \to 1^{-}$, which are sufficient to establish \Cref{thm:convergence}. This is still not straightforward, but can be obtained by using the asymptotic behavior of modified Bessel functions.

The authors would like to thank the anonymous referees for their careful reading of the manuscript and for providing helpful comments.

The paper is organized as follows. In \Cref{sec:preliminaries} we recall the hyperboloid model and study the modified Bessel function and its properties. \Cref{sec:kernel} is devoted to the proofs of \Cref{prop-asymptotic} and \Cref{thm-Bochner}. In \Cref{sec:extension}, we relate the heat semigroup to the extension problem \eqref{eq:ext-prob} and find the Poisson formula. Using the Poisson formula and the representation of the fractional $p$-Laplacian, we prove \Cref{thm:extension}. Finally, we prove the pointwise convergence result, \Cref{thm:convergence}, in \Cref{sec:convergence}. An auxiliary result can be found in \Cref{sec:appendix}.


\section{Preliminaries} \label{sec:preliminaries}


In this section, we recall the basics of the hyperbolic spaces and collect some facts about the modified Bessel function.

\subsection{The hyperbolic space}

There are several models for hyperbolic spaces, but let us focus on the hyperboloid model in this paper. The hyperboloid model is given by
\begin{equation*}
\mathbb{H}^{n} = \left\lbrace (x_{0}, \dots, x_{n}) \in \mathbb{R}^{n+1}: x_{0}^{2}-x_{1}^{2}-\cdots-x_{n}^{2} = 1, x_{0} > 0 \right\rbrace
\end{equation*}
with the Lorentzian metric $-\mathrm{d} x_{0}^{2}+\mathrm{d} x_{1}^{2}+\cdots+ \mathrm{d} x_{n}^{2}$ in $\mathbb{R}^{n+1}$. The Lorentzian metric induces the natural internal product
\begin{equation*}
[x, \xi] = x_{0} \xi_{0} - x_{1} \xi_{1} - \cdots - x_{n} \xi_{n}
\end{equation*}
on $\mathbb{H}^{n}$. Moreover, the distance between two points $x$ and $\xi$ is given by
\begin{equation*}
d(x, \xi) = \cosh^{-1}([x, \xi]).
\end{equation*}
Using the polar coordinates, $\mathbb{H}^{n}$ can also be realized as
\begin{equation*}
\mathbb{H}^{n} = \left\lbrace x=(\cosh r, \sinh r \, \omega) \in \mathbb{R}^{n+1}: r \geq 0, \omega \in \mathbb{S}^{n-1} \right\rbrace.
\end{equation*}
Then, the metric and the volume element are given by $\mathrm{d} r^{2} + \sinh^2 r \, \mathrm{d} \omega^{2}$ and $\sinh^{n-1} r \, \mathrm{d} r \, \mathrm{d} \omega$, respectively. 

\subsection{The modified Bessel function} \label{sec:Bessel}

The modified Bessel functions naturally appear in the study of hyperbolic geometry. In this paper, they are used to describe the kernel of the fractional $p$-Laplacian and the Poisson kernel. For this purpose, we recall the definition and some properties of the modified Bessel functions.

We call the ordinary differential equation
\begin{equation*}
\rho^{2} \frac{\mathrm{d}^{2} y}{\mathrm{d} \rho^{2}} + \rho \frac{\mathrm{d} y}{\mathrm{d} \rho} - (\rho^{2} + \nu^{2})y =0
\end{equation*}
the modified Bessel equation. The solutions are given by
\begin{equation*}
I_{\nu}(\rho) = \sum_{j=0}^{\infty} \frac{1}{j! \Gamma(\nu+j+1)}\left( \frac{\rho}{2} \right)^{2j+\nu} \quad \text{and} \quad K_{\nu}(\rho) = \frac{\pi}{2} \frac{I_{-\nu}(\rho)-I_{\nu}(\rho)}{\sin \nu \pi},
\end{equation*}
and they are called the {\it modified Bessel functions of the first and the second kind}, respectively. Since only $K_{\nu}$ appears in this work, we focus on the properties of $K_{\nu}$. This function has the following integral representation (see \cite[10.32.10]{OLBC10}):
\begin{equation} \label{eq:K-integral}
K_{\nu}(\rho) = \frac{1}{2}\left( \frac{1}{2}\rho \right)^{\nu} \int_{0}^{\infty} e^{-t-\frac{\rho^2}{4t}} t^{-\nu-1} \,\mathrm{d}t.
\end{equation}
The asymptotic behavior of $K_{\nu}$ is given by
\begin{equation} \label{eq:K-asymptotic}
\begin{split}
&K_{\nu}(\rho) \sim \frac{1}{2} \Gamma(\nu) \left( \frac{\rho}{2} \right)^{-\nu} \quad\text{as } \rho \to 0^{+}, \text{ for } \nu > 0, \text{ and} \\
&K_{\nu}(\rho) \sim \sqrt{\frac{\pi}{2\rho}} e^{-\rho} \quad\text{as } \rho \to \infty.
\end{split}
\end{equation}
Moreover, $K_{\nu}$ satisfies the following recurrence relations:
\begin{equation} \label{eq:K-recurrence}
K_{\nu}' = -K_{\nu-1}-\frac{\nu}{\rho}K_{\nu} \quad\text{and}\quad K_{\nu}' = -K_{\nu+1}+\frac{\nu}{\rho} K_{\nu}.
\end{equation}
We also recall that $K_{\nu}$ is increasing with respect to $\nu > 0$. For further properties of the modified Bessel functions, the reader may consult the handbook \cite{OLBC10}.

In the sequel, functions of the form $\rho^{-\nu}K_{\nu}(a\rho)$ with $\nu \in \mathbb{R}$ and $a > 0$ will appear frequently. For notational convenience, we define
\begin{equation} \label{eq:K-scr}
\widetilde{K}_{\nu, a}(\rho) := \rho^{-\nu} K_{\nu}(a\rho).
\end{equation}
Then, it follows from \eqref{eq:K-recurrence}
\begin{equation} \label{eq:scrK-recurrence}
-\partial_{\rho} (\widetilde{K}_{\nu, a}(f(\rho))) = a f'(\rho) f(\rho) \widetilde{K}_{\nu+1, a}(f(\rho))
\end{equation}
for any differentiable function $f : (0, \infty) \to (0, \infty)$.


\section{Nonlinear Bochner's formula} \label{sec:kernel}


The main goal of this section is to prove \Cref{thm-Bochner}. One of the crucial tools in the proof is the explicit formula for the heat kernel $h$ on hyperbolic spaces, which is given in \cite{GN98} as follows:
\begin{equation} \label{eq:hk-odd}
h(t, \rho) = \frac{1}{(2\pi)^{m}} \frac{1}{(4\pi t)^{1/2}} \left( \frac{-\partial_{\rho}}{\sinh \rho} \right)^{m} e^{-m^{2} t-\frac{\rho^2}{4t}}
\end{equation}
when $n = 2m+1 \geq 1$ is odd, and
\begin{equation} \label{eq:hk-even}
h(t, \rho) = \frac{1}{2(2\pi)^{m+1/2}} t^{-3/2} e^{-\frac{(2m-1)^2}{4}t} \left( \frac{-\partial_{\rho}}{\sinh \rho} \right)^{m-1} \int_{\rho}^{\infty} \frac{r e^{-\frac{r^2}{4t}}}{\sqrt{\cosh r-\cosh \rho}} \,\mathrm{d}r
\end{equation}
when $n=2m \geq 2$ is even. Note that the Cauchy problem \eqref{eq:Cauchy} has the unique solution
\begin{equation}\label{eq-w}
w(x,t) = \int_{\mathbb{H}^{n}} h(t, d(x, \xi)) f(\xi) \,\mathrm{d}\xi,
\end{equation}
provided that $f$ is a bounded continuous function. In other words, the heat semigroup $e^{t\Delta_{\mathbb{H}^n}}[f]$ is given by \eqref{eq-w}.

Before proving \Cref{thm-Bochner}, we provide a series of auxiliary lemmas required for its proof. In the course of the argument, we also establish \Cref{prop-asymptotic}.

\begin{lemma} \label{lem:recurrence}
Let $\nu > 1/2$, $a \geq 1/2$, and $y \geq 0$. For $m \in \mathbb{N} \cup \lbrace 0 \rbrace$ define
\begin{equation*}
F_{m}(r) := \frac{\sinh r}{\sqrt{\cosh r - \cosh\rho}} \left( \frac{-\partial_{r}}{\sinh r} \right)^{m} \widetilde{K}_{\nu, a}\left( \sqrt{r^{2}+y^{2}} \right),
\end{equation*}
where $\widetilde{K}_{\nu, a}$ is the function given in \eqref{eq:K-scr}. Then, $F_{m}$ is integrable on $(\rho, \infty)$ and satisfies
\begin{equation} \label{eq:recurrence-F}
\left( \frac{-\partial_{\rho}}{\sinh \rho} \right) \int_{\rho}^{\infty} F_{m}(r) \,\mathrm{d}r = \int_{\rho}^{\infty} F_{m+1}(r) \,\mathrm{d}r
\end{equation}
for all $\rho>0$ and $m \in \mathbb{N} \cup \lbrace 0 \rbrace$.
\end{lemma}

\begin{proof}
Note that for any $j \geq 1$
\begin{equation*}
-\partial_{r} \left( \frac{(e^{r}+e^{-r})^{j-1}}{(e^{r}-e^{-{r}})^{j}} \right) = j\frac{(e^{r}+e^{-r})^{j}}{(e^{r}-e^{-r})^{j+1}} - (j-1)\frac{(e^{r}+e^{-r})^{j-2}}{(e^{r}-e^{-r})^{j-1}}.
\end{equation*}
Therefore, all derivatives of $\frac{1}{\sinh r}$ (and $\frac{r}{\sinh r}$) have the same asymptotic behavior as $e^{-r}$ (and $re^{-r}$, respectively) as $r \to \infty$. Hence, $F_{m}(r) \sim r^{-\nu-1/2}e^{(1/2-m-a)r}$ as $r \to \infty$, which shows that the function $F_{m}$ is integrable.

Using the integration by parts, we have
\begin{equation*}
\begin{split}
\int_{\rho}^{\infty} F_{m}(r) \,\mathrm{d}r
&= \int_{\rho}^{\infty} 2\partial_{r} \left( \sqrt{\cosh r - \cosh \rho} \right) \left( \frac{-\partial_r}{\sinh r} \right)^{m} \widetilde{K}_{\nu, a} \left( \sqrt{r^{2}+y^{2}} \right)\mathrm{d}r \\
&= \int_{\rho}^{\infty} 2 \sinh r \sqrt{\cosh r - \cosh \rho} \left( \frac{-\partial_{r}}{\sinh r} \right)^{m+1} \widetilde{K}_{\nu, a} \left( \sqrt{r^{2}+y^{2}} \right)\mathrm{d}r.
\end{split}
\end{equation*}
Thus, the recurrence relation \eqref{eq:recurrence-F} follows by applying the Leibniz integral rule.
\end{proof}

\begin{lemma} \label{lem:positivity}
Let $\nu > 1/2$, $a \geq 1/2$, and $m \in \mathbb{N} \cup \lbrace 0 \rbrace$. Then, the function
\begin{equation*}
\rho \mapsto \left( \frac{-\partial_{\rho}}{\sinh \rho} \right)^{m} \widetilde{K}_{\nu,a}(\rho)
\end{equation*}
is positive.
\end{lemma}

\begin{proof}
Using the formula \eqref{eq:K-integral} and change of variables, we have
\begin{equation*}
\widetilde{K}_{\nu,a}(\rho) = \frac{a^{\nu}}{2^{\nu+1}} \int_{0}^{\infty} e^{-t-\frac{(a\rho)^2}{4t}} t^{-\nu-1} \,\mathrm{d}t = \frac{1}{2(2a)^{\nu}} \int_{0}^{\infty} \frac{1}{t^{1/2}} e^{-a^{2}t-\frac{\rho^{2}}{4t}} \frac{\mathrm{d}t}{t^{\nu+1/2}}.
\end{equation*}
Thus, recalling the expression of the heat kernel \eqref{eq:hk-odd} for odd-dimensional case, we obtain
\begin{equation*}
\left( \frac{-\partial_{\rho}}{\sinh \rho} \right)^{m} \widetilde{K}_{\nu,a}(\rho) = \int_{0}^{\infty} e^{(m^{2}-a^{2})t} h(t, \rho) \frac{\mathrm{d}t}{t^{\nu+1/2}}.
\end{equation*}
The conclusion follows from the positivity of the heat kernel $h$.
\end{proof}

As a consequence of \Cref{lem:positivity}, we obtain the positivity of the kernel $\mathcal{K}_{n,s,p}$.

\begin{corollary} \label{cor:nonnegativity}
Let $n \in \mathbb{N}$, $0<s<1$, and $p > 1$. The kernel $\mathcal{K}_{n,s,p}$ is positive.
\end{corollary}

In the following lemma, we consider a class of functions more general than $\mathcal{K}_{n, s, p}$ to allow for a broader applicability in later results. Note that \Cref{prop-asymptotic} follows from \Cref{lem-asymptotic} with $\nu=\frac{1+sp}{2}$ and $a=\frac{n-1}{2}$, together with \eqref{eq:K-asymptotic}.

\begin{lemma}\label{lem-asymptotic}
Let $\nu>1/2$, $a \geq 1/2$, and $y \geq 0$. Let
\begin{equation*}
K(\rho) = \left( \frac{-\partial_\rho}{\sinh \rho} \right)^{\frac{n-1}{2}} \widetilde{K}_{\nu,a}(\rho)
\end{equation*}
when $n \geq 3$ is odd and
\begin{equation*}
K(\rho) = \int_{\rho}^{\infty} \frac{\sinh r}{\sqrt{\cosh r - \cosh \rho}} \left( \frac{-\partial_r}{\sinh r} \right)^{\frac{n}{2}} \widetilde{K}_{\nu,a}(r) \mathrm{d}r
\end{equation*}
when $n \geq 2$ is even. Then, there exist $c, C>0$ such that
\begin{equation*}
c\rho^{-\nu} (\sinh\rho)^{-\frac{n-1}{2}} K_{\frac{n-1}{2}+\nu}(a\rho) \leq K(\rho) \leq C\rho^{-\nu} (\sinh\rho)^{-\frac{n-1}{2}} K_{\frac{n-1}{2}+\nu}(a\rho)
\end{equation*}
for all $\rho > 0$.
\end{lemma}

In order to prove \Cref{lem-asymptotic}, we need the following lemma.

\begin{lemma} \label{lem:even_dim}
  Let $a>0$ and $\nu > -\frac{n-1}{2}$. Then
  \begin{align*}
    \int^{\infty}_{\rho} \frac{\sinh^{-n/2+1} r}{\sqrt{\cosh r-\cosh \rho}} r^{-\nu}K_{n/2 + \nu} (ar) \,\mathrm{d} r \sim \sqrt{\frac{\pi}{2}} \frac{\Gamma(\nu+\frac{n-1}{2})}{\Gamma(\nu+\frac{n}{2})} \rho^{-\nu}\sinh^{-n/2 +1} (\rho)K_{n/2+\nu} (a\rho) 
  \end{align*}
  as $\rho \to 0^{+}$ up to dimensional constants.
\end{lemma}

\begin{proof}
By the change of variables $r=\rho t$, we have
  \begin{align*}
    &\int_\rho^\infty \frac{1}{\sqrt{\cosh r-\cosh \rho}} \frac{r^{-\nu}}{\rho^{-\nu}} \frac{\sinh^{-n/2+1}r}{\sinh^{-n/2+1}\rho} \frac{K_{n/2 + \nu} (ar)}{K_{n/2 + \nu} (a\rho)} \,\mathrm{d} r \\
    &= \int ^ \infty_1 \frac{\rho t^{-\nu}}{\sqrt{\cosh(\rho t)-\cosh \rho}} \frac{\sinh^{-n/2+1}(\rho t)}{\sinh^{-n/2+1}\rho}\frac{K_{n/2 + \nu} (a\rho t)}{K_{n/2 + \nu} (a\rho)} \,\mathrm{d}t.
  \end{align*}
  We define for each $\rho \in (0,1)$ a function $f_\rho$ by
  \begin{align*}
    f_\rho(t) = \frac{\rho t^{-\nu}}{\sqrt{\cosh(\rho t)-\cosh \rho}} \frac{\sinh^{-n/2+1}(\rho t)}{\sinh^{-n/2+1}\rho}\frac{K_{n/2 + \nu} (a\rho t)}{K_{n/2 + \nu} (a\rho)}, \quad t \in (1, \infty).
  \end{align*}
Note that
  \begin{align*}
    \frac{\cosh(\rho t) - \cosh\rho}{\rho^2}   \ge \frac{1}{2}(t^2-1) \quad\text{and}\quad \sinh(\rho t) \ge (\sinh\rho)t.
  \end{align*}
  Moreover, by \cite[Equation (2.17)]{JB96}, we have
  \begin{align*}
    \frac{K_{n/2 + \nu} (a\rho t)}{K_{n/2 + \nu} (a\rho)} \le t^{-\frac{n}{2}-\nu}.
  \end{align*}
  Thus, $f_\rho$ is bounded from above by a function
  \begin{align*}
    f(t) = \frac{t^{-n-2\nu +1 }}{\sqrt{(t^2-1)/2}},
  \end{align*}
  which is integrable on $(0,\infty)$. Indeed, by the change of variables $t^2-1 = \tau$, we obtain
  \begin{align*}
    \int_1^\infty f(t) \,\mathrm{d} t = \frac{1}{\sqrt{2}} \int_0 ^ \infty \frac{\tau^{-1/2}}{(1+\tau)^{n/2+\nu}} \,\mathrm{d} \tau  = \frac{1}{\sqrt{2}} B\left(\frac{1}{2}, \nu+\frac{n-1}{2}   \right)= \sqrt{\frac{\pi}{2}} \frac{\Gamma(\nu+\frac{n-1}{2})}{\Gamma(\nu+\frac{n}{2})} < \infty,
  \end{align*}
  where $B$ is Euler's Beta Integral (see \cite[5.12.3]{OLBC10}).
  
  For fixed $t \in (1, \infty)$, we have
  \begin{align*}
    \frac{\cosh(\rho t) - \cosh\rho}{\rho^2} \to \frac{1}{2}(t^2-1), \quad \frac{\sinh(\rho t)}{\sinh\rho} \to t, \quad \mbox{and} \quad \frac{K_{n/2 + \nu} (a\rho t)}{K_{n/2 + \nu} (a\rho)}  \to t^{-n/2 - \nu }
  \end{align*}
  as $\rho \to 0^{+}$. Hence, we obtain $\lim_{\rho \to 0} f_\rho(t) = f(t)$. Therefore, the Lebesgue dominated convergence theorem concludes the lemma. 
\end{proof}

\begin{proof}[Proof of \Cref{lem-asymptotic}]
We prove the odd-dimensional case $n=2m+1$ first. We already provided a simple way to chase the asymptotic behavior of operator $\left ( \frac{-\partial_r}{\sinh r} \right )^m$ at the former part of proof of \Cref{lem:recurrence}. By applying the similar argument one can shows that $K(\rho) \sim \rho^{-2(m+\nu)}$ as $\rho \to 0^+$ and $K(\rho) \sim \rho^{-\frac{1}{2}-\nu}e^{-(a+m)\rho}$ as $\rho \to \infty$. In other words, $K(\rho) \sim \rho^{-\nu} (\sinh\rho)^{-m} K_{m+\nu}(a\rho)$ as $\rho \to 0^+$ and $\rho \to \infty$ by \eqref{eq:K-asymptotic}. Since $K$ is positive by \Cref{lem:positivity} and continuous in $(0, \infty)$, there exist $c, C>0$ such that
\begin{equation*}
    c\rho^{-\nu} (\sinh\rho)^{-m} K_{m+\nu}(a\rho) \leq K(\rho) \leq C\rho^{-\nu} (\sinh\rho)^{-m} K_{m+\nu}(a\rho)
\end{equation*}
for all $\rho >0$.

Let us consider the even-dimensional case $n=2m$. By using the result for the odd-dimensional case, we obtain for $\rho$ sufficiently close to $\infty$ that
\begin{align*}
K(\rho)
&\le C \int_\rho^\infty \frac{e^r}{\sqrt{\sinh \frac{r-\rho}{2}} \sqrt{\sinh \frac{r+\rho}{2}}} r^{-\frac{1}{2}-\nu} e^{-(a+m) r} \,\mathrm{d}r \\
& \le C \frac{1}{\sqrt{\sinh \rho}}\int _0^\infty \frac{1}{\sqrt{\sinh \frac{t}{2}}} (t+\rho)^{-\frac{1}{2}-\nu} e^{-(a+m-1)(t+\rho)} \,\mathrm{d}t \\
& \le C \rho^{-\frac{1}{2}-\nu} e^{-(a+\frac{2m-1}{2})\rho} \int_0^\infty \frac{1}{\sqrt{\sinh \frac{t}{2}}} dt \\
&\leq C \rho^{-\frac{1}{2}-\nu} e^{-(a+\frac{n-1}{2})\rho}
\end{align*}
and
\begin{align*}
K(\rho)
&= \int_\rho^\infty 2\sinh r \sqrt{\cosh r - \cosh \rho} \left(\frac{-\partial_r}{\sinh r} \right)^{\frac{n+2}{2}} \widetilde{K}_{\nu, a}(r) \,\mathrm{d}r \\
& \ge c \int_\rho^\infty e^r\sqrt{\sinh \frac{r-\rho}{2}} \sqrt{\sinh \frac{r+\rho}{2}} r^{-\frac{1}{2}-\nu} e^{-(a+m+1) r} \,\mathrm{d}r \\
& \ge c \sqrt{\sinh \rho} \int_0^\infty \sqrt{\sinh \frac{t}{2}} (t+\rho)^{-\frac{1}{2}-\nu} e^{-(a+m)(t+\rho)} \,\mathrm{d}t \\
& \ge c \rho^{-\frac{1}{2}-\nu} e^{-(a+\frac{2m-1}{2})\rho} \int _0 ^\infty \sqrt{\sinh \frac{t}{2}} (1+t)^{-\frac{1}{2}-\nu}e^{-(a+m)t} \,\mathrm{d}t \\
&\geq c \rho^{-\frac{1}{2}-\nu} e^{-(a+\frac{n-1}{2})\rho}
  \end{align*}
for some constants $c, C>0$. In other words, $K$ is comparable to $\rho^{-\nu} (\sinh\rho)^{-\frac{n-1}{2}} K_{\frac{n-1}{2}+\nu}(a\rho)$ as $\rho \to \infty$.

On the other hand, $K(\rho)$ is comparable to $\rho^{-2\nu-(n-1)}$, or to $\rho^{-\nu} (\sinh\rho)^{-\frac{n-1}{2}} K_{\frac{n-1}{2}+\nu}(a\rho)$, as $\rho \to 0^+$. Indeed, by using the result for the odd-dimensional case once again, we infer that $K(\rho)$ is comparable to
\begin{equation*}
\int_\rho^\infty \frac{\sinh^{-m+1} r}{\sqrt{\cosh r - \cosh \rho}} r^{-\nu}K_{m+\nu}(ar) \,\mathrm{d}r.
\end{equation*}
By \Cref{lem:even_dim}, it is in turn comparable to
\begin{equation*}
\rho^{-\nu} \sinh^{-m+1}(\rho) K_{m+\nu}(a\rho)
\end{equation*}
near $\rho=0^+$. Since $K$ is positive and continuous in $(0, \infty)$, the desired result follows.
\end{proof}

Let us now prove \Cref{thm-Bochner} using the heat kernel and previous lemmas.

\begin{proof} [Proof of \Cref{thm-Bochner}]
Let $\varepsilon > 0$ and define $g_{\varepsilon}(x,\xi) = \Phi_{p}(u(x)-u(\xi)) \chi_{d(x,\xi) > \varepsilon}$. The heat semigroup associated to $g_{\varepsilon}(x, \cdot)$ is given by
\begin{equation*}
e^{t\Delta_{\mathbb{H}^{n}}} [g_{\varepsilon}(x,\cdot)](x) = \int_{\mathbb{H}^{n}} \frac{1}{(2\pi)^{m}} \frac{1}{(4\pi t)^{1/2}} \left( \left( \frac{-\partial_{\rho}}{\sinh \rho} \right)^{m} e^{-m^{2} t-\frac{\rho^2}{4t}} \right) g_{\varepsilon}(x,\xi) \,\mathrm{d}\xi
\end{equation*}
when $n=2m+1 \geq 3$ is odd and
\begin{equation*}
e^{t\Delta_{\mathbb{H}^{n}}} [g_{\varepsilon}(x,\cdot)](x) = \int_{\mathbb{H}^{n}} \frac{t^{-3/2} e^{-\frac{(2m-1)^2}{4}t}}{2(2\pi)^{m+1/2}} \left( \frac{-\partial_{\rho}}{\sinh \rho} \right)^{m-1} \int_{\rho}^{\infty} \frac{r e^{-\frac{r^2}{4t}} \,\mathrm{d}r}{\sqrt{\cosh r-\cosh \rho}} \, g_{\varepsilon}(x, \xi) \,\mathrm{d}\xi
\end{equation*}
when $n=2m \geq 2$ is even, where $\rho = d(x,\xi)$. We will prove
\begin{equation} \label{eq:truncated}
c_{n, s, p} \int_{d(x, \xi) > \varepsilon} \Phi_p(u(x)-u(\xi)) \mathcal{K}_{n,s,p}(d(x,\xi)) \,\mathrm{d}\xi = C_{2} \int_{0}^{\infty} e^{t\Delta_{\mathbb{H}^{n}}}[g_{\varepsilon}(x,\cdot)](x) \frac{\mathrm{d}t}{t^{1+\frac{sp}{2}}}
\end{equation}
in both cases. 

Let us first consider the odd-dimensional case. We fix $\delta > 0$ and integrate the heat semigroup with respect to the singular measure $t^{-1-\frac{sp}{2}} \,\mathrm{d}t$ over the interval $(\delta, \infty)$ to obtain
\begin{equation} \label{eq:heat-semigroup-odd}
\begin{split}
&C_{2} \int_{\delta}^{\infty} e^{t\Delta_{\mathbb{H}^{n}}}[g_{\varepsilon}(x,\cdot)](x) \frac{\mathrm{d}t}{t^{1+\frac{sp}{2}}} \\
&= C_{2} \int_{\delta}^{\infty} \int_{\mathbb{H}^{n}} \frac{1}{(2\pi)^{m}} \frac{1}{(4\pi t)^{1/2}} \left( \left( \frac{-\partial_{\rho}}{\sinh \rho} \right)^{m} e^{-m^{2} t-\frac{\rho^2}{4t}} \right) g_{\varepsilon}(x,\xi) \,\mathrm{d}\xi \, \frac{\mathrm{d}t}{t^{1+\frac{sp}{2}}}.
\end{split}
\end{equation}
Note that this expression is well defined since $|e^{t\Delta_{\mathbb{H}^n}}[g_\varepsilon(x, \cdot)](x)| \leq C\|u\|_{L^\infty}^{p-1}$ for some constant $C>0$. By applying Fubini's theorem, we have
\begin{equation}\label{eq:heat-semigroup-odd2}
\begin{split}
&C_{2} \int_{\delta}^{\infty} \int_{\mathbb{H}^{n}} \frac{1}{(2\pi)^{m}} \frac{1}{(4\pi t)^{1/2}} \left( \left( \frac{-\partial_{\rho}}{\sinh \rho} \right)^{m} e^{-m^{2} t-\frac{\rho^2}{4t}} \right) g_{\varepsilon}(x,\xi) \,\mathrm{d}\xi \, \frac{\mathrm{d}t}{t^{1+\frac{sp}{2}}} \\
&= \frac{C_{2}}{(2\pi)^{m}(4\pi)^{1/2}} \int_{\mathbb{H}^{n}} \int_{\delta}^{\infty} \left( \left( \frac{-\partial_{\rho}}{\sinh \rho} \right)^{m} e^{-m^{2} t-\frac{\rho^2}{4t}}t^{-\frac{3+sp}{2}} \right) \, \mathrm{d}t \,g_{\varepsilon}(x,\xi) \,\mathrm{d}\xi.
\end{split}
\end{equation}
Furthermore, since all partial derivatives of $e^{-m^{2} t-\frac{\rho^2}{4t}}t^{-\frac{3+sp}{2}}$ with respect to $\rho$ are integrable over the interval $(\delta, \infty)$, the dominated convergence theorem shows that
\begin{equation}\label{eq:heat-semigroup-odd3}
\begin{split}
&\int_{\mathbb{H}^{n}} \int_{\delta}^{\infty} \left( \left( \frac{-\partial_{\rho}}{\sinh \rho} \right)^{m} e^{-m^{2} t-\frac{\rho^2}{4t}}t^{-\frac{3+sp}{2}} \right) \, \mathrm{d}t \,g_{\varepsilon}(x,\xi) \,\mathrm{d}\xi \\
&= \int_{\mathbb{H}^{n}} \left( \frac{-\partial_{\rho}}{\sinh \rho} \right)^{m} \left( \int_{\delta}^{\infty} e^{-m^{2} t-\frac{\rho^2}{4t}} t^{-\frac{3+sp}{2}} \, \mathrm{d}t \right) g_{\varepsilon}(x,\xi) \,\mathrm{d}\xi.
\end{split}
\end{equation}
Note that the function $e^{-m^{2}t-\frac{\rho^2}{4t}} t^{-\frac{3+sp}{2}}$ is integrable on $(0, \infty)$. Indeed, the formula \eqref{eq:K-integral} and the change of variables show
\begin{equation} \label{eq:K-cv}
\int_{0}^{\infty} e^{-m^{2}t-\frac{\rho^2}{4t}} t^{-\frac{3+sp}{2}} \,\mathrm{d}t = m^{1+sp} \int_{0}^{\infty} e^{-t-\frac{(m\rho)^2}{4t}} t^{-\frac{3+sp}{2}} \,\mathrm{d}t = 2(2m)^{\frac{1+sp}{2}} \widetilde{K}_{\frac{1+sp}{2},m}(\rho).
\end{equation}
Thus, \eqref{eq:truncated} in the odd-dimensional case follows by combining \eqref{eq:heat-semigroup-odd}--\eqref{eq:K-cv} and passing to the limit $\delta \searrow 0$.

We next consider the even-dimensional case. Similarly to as in the odd-dimensional case, we obtain
\begin{equation*}
\begin{split}
&C_{2} \int_{\delta}^{\infty} e^{t\Delta_{\mathbb{H}^{n}}}[g_{\varepsilon}(x,\cdot)](x) \frac{\mathrm{d}t}{t^{1+\frac{sp}{2}}} \\
&= C_{2} \int_{\delta}^{\infty} \int_{\mathbb{H}^{n}} \frac{t^{-3/2} e^{-\frac{(2m-1)^2}{4}t}}{2(2\pi)^{m+1/2}} \left( \frac{-\partial_{\rho}}{\sinh \rho} \right)^{m-1} \int_{\rho}^{\infty} \frac{r e^{-\frac{r^{2}}{4t}}}{\sqrt{\cosh r - \cosh \rho}} \,\mathrm{d}r \, g_{\varepsilon}(x,\xi) \,\mathrm{d}\xi \, \frac{\mathrm{d}t}{t^{1+\frac{sp}{2}}} \\
&= \frac{C_{2}}{2(2\pi)^{m+1/2}} \int_{\mathbb{H}^{n}} \left( \frac{-\partial_{\rho}}{\sinh \rho} \right)^{m-1} \int_{\rho}^{\infty} \left( \int_{\delta}^{\infty} e^{-\frac{(2m-1)^2}{4} t-\frac{r^2}{4t}} t^{-\frac{5+sp}{2}} \, \mathrm{d}t \right) \frac{r \,\mathrm{d}r \, g_{\varepsilon}(x, \xi) \,\mathrm{d}\xi}{\sqrt{\cosh r-\cosh \rho}}.
\end{split}
\end{equation*}
Moreover, we have from \eqref{eq:K-integral} and \eqref{eq:scrK-recurrence}
\begin{equation*}
\begin{split}
\int_{0}^{\infty} e^{-\frac{(2m-1)^2}{4} t-\frac{r^2}{4t}} t^{-\frac{5+sp}{2}} \, \mathrm{d}t
&= 2 (2m-1)^{\frac{3+sp}{2}} \widetilde{K}_{\frac{3+sp}{2}, \frac{2m-1}{2}}(r) \\
&= 4 (2m-1)^{\frac{1+sp}{2}} \left( \frac{-\partial_{r}}{r} \right) \widetilde{K}_{\frac{1+sp}{2}, \frac{2m-1}{2}}(r).
\end{split}
\end{equation*}
Thus, we deduce
\begin{equation*}
\begin{split}
&C_{2} \int_{0}^{\infty} e^{t\Delta_{\mathbb{H}^{n}}}[g_{\varepsilon}(x,\cdot)](x) \frac{\mathrm{d}t}{t^{1+\frac{sp}{2}}} \\
&= c_{n, s, p} C_{1} \int_{\mathbb{H}^{n}} \left( \frac{-\partial_\rho}{\sinh \rho} \right)^{m-1} \int_{\rho}^{\infty} \frac{(-\partial_r) \widetilde{K}_{\frac{1+sp}{2},\frac{2m-1}{2}}(r)}{\sqrt{\pi}\sqrt{\cosh r-\cosh \rho}} \,\mathrm{d}r \, g_{\varepsilon}(x, \xi) \,\mathrm{d}\xi \\
&= c_{n, s, p} C_{1} \int_{\mathbb{H}^{n}} \int_{\rho}^{\infty} \frac{\sinh r}{\sqrt{\pi}\sqrt{\cosh r-\cosh \rho}} \left(\frac{-\partial_r}{\sinh r} \right)^{m} \widetilde{K}_{\frac{1+sp}{2},\frac{2m-1}{2}}(r) \,\mathrm{d}r \, g_{\varepsilon}(x, \xi) \,\mathrm{d}\xi,
\end{split}
\end{equation*}
where we used \Cref{lem:recurrence} $(m-1)$-times with $\nu = \frac{1+sp}{2}$, $a=\frac{2m-1}{2}$, and $y = 0$ in the last equality. This proves \eqref{eq:truncated} in the even-dimensional case.

On the one hand, the integral in the right-hand side of \eqref{eq:truncated} converges to the Cauchy principal value
\begin{equation*}
\mathrm{P.V.} \int_{\mathbb{H}^{n}} \Phi_p(u(x)-u(\xi)) \mathcal{K}_{n,s,p}(d(x,\xi)) \,\mathrm{d}\xi
\end{equation*}
as $\varepsilon \searrow 0$. For the left-hand side of \eqref{eq:truncated}, on the other hand, we need to estimate
\begin{equation*}
A:= \int_{0}^{\infty} e^{t\Delta_{\mathbb{H}^{n}}}[\Phi_{p}(u(x)-u(\cdot))](x) \frac{\mathrm{d}t}{t^{1+\frac{sp}{2}}} - \int_{0}^{\infty} e^{t\Delta_{\mathbb{H}^{n}}}[g_{\varepsilon}(x,\cdot)](x) \frac{\mathrm{d}t}{t^{1+\frac{sp}{2}}}.
\end{equation*}
Proceeding as above, we have
\begin{equation*}
|A| \lesssim \left| \mathrm{P.V.} \int_{d(x, \xi)\leq \varepsilon} \Phi_{p}(u(x)-u(\xi)) \mathcal{K}_{n,s,p}(d(x, \xi)) \,\mathrm{d}\xi \right|.
\end{equation*}
Thus, applying \Cref{lem:appendix} to $K = \mathcal{K}_{n,s,p}\chi_{\{d(x, \xi)\leq \varepsilon\}}$ yields
\begin{equation} \label{eq:kernel-zero}
|A| \lesssim \int_{d(x,\xi) \leq \varepsilon} \rho^{\alpha} \mathcal{K}_{n,s,p}(\rho) \,\mathrm{d}y \lesssim \int_{0}^{\varepsilon} \rho^{\alpha} \mathcal{K}_{n,s,p}(\rho) \sinh^{n-1}\rho \,\mathrm{d}\rho,
\end{equation}
where $\alpha = 2p-2$ when $p \in (\frac{2}{2-s}, 2)$ and $\alpha = p$ when $p \in (1, \frac{2}{2-s}] \cup [2, \infty)$. Recall that $\mathcal{K}_{n,s,p}$ is positive by \Cref{cor:nonnegativity}. Moreover, \Cref{lem-asymptotic} (or \Cref{prop-asymptotic}) shows that the function $\rho^{\alpha} \mathcal{K}_{n,s,p}(\rho) \sinh^{n-1}\rho$ is integrable near zero and hence the right-hand side of \eqref{eq:kernel-zero} converges to zero as $\varepsilon \searrow 0$. Therefore, the left-hand side of \eqref{eq:truncated} converges to that of \eqref{eq:pw-representation} as $\varepsilon \searrow 0$.
\end{proof}


\section{Extension problem} \label{sec:extension}


In this section, we prove \Cref{thm:extension}, which provides another representation of the fractional $p$-Laplacian on hyperbolic spaces. We first find the Poisson formula and relate the heat semigroup to the extension problem \eqref{eq:ext-prob}.

\begin{lemma} \label{lem:Poisson}
Let $n \geq 2$, $s \in (0,1)$, and $p > 1$. If $f \in C_{b}(\mathbb{H}^{n})$, then
\begin{equation*}
E_{s, p}[f](x, y) := \int_{\mathbb{H}^{n}} P(d(x,\xi), y) f(\xi) \,\mathrm{d}\xi
\end{equation*}
is a solution of the extension problem \eqref{eq:ext-prob}, where $P(\rho, y)$ is the Poisson kernel given by
\begin{equation*}
P(\rho, y) = C_{4} \, y^{sp} \left( \frac{-\partial_\rho}{\sinh \rho} \right)^{\frac{n-1}{2}} \widetilde{K}_{\frac{1+sp}{2},\frac{n-1}{2}} \left( \sqrt{\rho^{2}+y^{2}} \right)
\end{equation*}
when $n \geq 3$ odd and
\begin{equation*}
P(\rho, y) = C_{4} \, y^{sp} \int_{\rho}^{\infty} \frac{\sinh r}{\sqrt{\pi}\sqrt{\cosh r - \cosh \rho}} \left( \frac{-\partial_{r}}{\sinh r} \right)^{\frac{n}{2}} \widetilde{K}_{\frac{1+sp}{2},\frac{n-1}{2}}\left( \sqrt{r^{2}+y^{2}} \right)\, \mathrm{d}r
\end{equation*}
when $n \geq 2$ even; here, $\widetilde{K}_{\nu,a}$ is the function given in \eqref{eq:K-scr} and
\begin{equation*}
C_{4} = \frac{1}{2^{\frac{n-3}{2}}\pi^{\frac{n}{2}} \Gamma(\frac{sp}{2})} \left( \frac{n-1}{4} \right)^{\frac{1+sp}{4}}.
\end{equation*}
Moreover, $E_{s, p}[f]$ has an alternative representation
\begin{equation} \label{eq:U}
E_{s, p}[f](x, y) = \frac{y^{sp}}{2^{sp}\Gamma(\frac{sp}{2})} \int_{0}^{\infty} e^{t\Delta_{\mathbb{H}^{n}}}[f](x) e^{-\frac{y^2}{4t}} \frac{\mathrm{d}t}{t^{1+\frac{sp}{2}}}.
\end{equation}
\end{lemma}

\begin{proof}
For each $x \in \mathbb{H}^{n}$ and $y > 0$, we define $V(x,y)$ by the function given in the right-hand side of \eqref{eq:U}. Then, we have
\begin{equation*}
V(x,y) = \frac{y^{sp}}{2^{sp}\Gamma(\frac{sp}{2})} \int_{0}^{\infty} \int_{\mathbb{H}^{n}} h(t, \rho) f(\xi) \,\mathrm{d}\xi \, e^{-\frac{y^{2}}{4t}} \frac{\mathrm{d}t}{t^{1+\frac{sp}{2}}},
\end{equation*}
where $\rho = d(x, \xi)$. Recalling the expression \eqref{eq:hk-odd} for the heat kernel $h(t, \rho)$ and using \eqref{eq:K-integral}, we obtain
\begin{equation*}
\begin{split}
V(x,y)
&= \frac{y^{sp}}{2^{sp}\Gamma(\frac{sp}{2})} \int_{0}^{\infty} \int_{\mathbb{H}^{n}} \frac{1}{(2\pi)^{m}} \frac{1}{(4\pi t)^{1/2}} \left( \left( \frac{-\partial_{\rho}}{\sinh \rho} \right)^{m} e^{-m^{2} t-\frac{\rho^{2}+y^{2}}{4t}} \right) f(\xi) \,\mathrm{d}\xi \, \frac{\mathrm{d}t}{t^{1+sp/2}} \\
&= \int_{\mathbb{H}^{n}} \frac{y^{sp}}{2^{sp}\Gamma(\frac{sp}{2})} \frac{1}{(2\pi)^{m}} \frac{1}{(4\pi)^{1/2}} \left( \frac{-\partial_{\rho}}{\sinh \rho} \right)^{m} \left( \int_{0}^{\infty} e^{-m^{2} t-\frac{\rho^{2}+y^{2}}{4t}} t^{-\frac{3+sp}{2}} \,\mathrm{d}t \right) f(\xi) \,\mathrm{d}\xi \\
&= \int_{\mathbb{H}^{n}} P(d(x,\xi), y) f(\xi) \,\mathrm{d}\xi
\end{split}
\end{equation*}
when $n=2m+1$ is odd. If $n=2m$ is even, then we use \eqref{eq:hk-even} instead of \eqref{eq:hk-odd} to have
\begin{equation*}
\begin{split}
V(x,y)
&= \frac{y^{sp}}{2^{sp}\Gamma(\frac{sp}{2})} \int_{0}^{\infty} \int_{\mathbb{H}^{n}} \frac{t^{-\frac{5+sp}{2}} e^{-\frac{(2m-1)^2}{4}t}}{2(2\pi)^{m+1/2}} \left( \frac{-\partial_{\rho}}{\sinh \rho} \right)^{m-1} \int_{\rho}^{\infty} \frac{r e^{-\frac{r^{2}+y^{2}}{4t}} \,\mathrm{d}r}{\sqrt{\cosh r-\cosh \rho}} \, f(\xi) \,\mathrm{d}\xi \, \mathrm{d}t \\
&= \frac{y^{sp}}{2^{sp}\Gamma(\frac{sp}{2})} \int_{\mathbb{H}^{n}} \left( \frac{-\partial_{\rho}}{\sinh \rho} \right)^{m-1} \int_{\rho}^{\infty} \int_{0}^{\infty} \frac{t^{-\frac{5+sp}{2}} e^{-\frac{(2m-1)^2}{4}t}}{2(2\pi)^{m+1/2}} \frac{r e^{-\frac{r^{2}+y^{2}}{4t}} \,\mathrm{d}t}{\sqrt{\cosh r-\cosh \rho}} \, \mathrm{d}r \,f(\xi) \, \mathrm{d}\xi.
\end{split}
\end{equation*}
Moreover, using \eqref{eq:K-integral} we compute
\begin{equation*}
\begin{split}
\int_{0}^{\infty} e^{-\frac{(2m-1)^2}{4} t-\frac{r^{2}+y^{2}}{4t}} t^{-\frac{5+sp}{2}} \, \mathrm{d}t
&= 2 (2m-1)^{\frac{3+sp}{2}} \widetilde{K}_{\frac{3+sp}{2}, \frac{2m-1}{2}}\left( \sqrt{r^{2}+y^{2}} \right) \\
&= 4(2m-1)^{\frac{1+sp}{2}} \left( \frac{-\partial_{r}}{r} \right) \widetilde{K}_{\frac{1+sp}{2}, \frac{2m-1}{2}} \left( \sqrt{r^{2}+y^{2}} \right).
\end{split}
\end{equation*}
Therefore, we obtain
\begin{equation*}
\begin{split}
V(x,y)
&= C_{4} \, y^{sp} \int_{\mathbb{H}^{n}} \left( \frac{-\partial_{\rho}}{\sinh \rho} \right)^{m-1} \int_{\rho}^{\infty} \frac{(-\partial_{r}) \widetilde{K}_{\frac{1+sp}{2}, \frac{2m-1}{2}}\left(\sqrt{r^{2}+y^{2}} \right)}{\sqrt{\pi}\sqrt{\cosh r-\cosh \rho}} \, \mathrm{d}r \,f(\xi) \, \mathrm{d}\xi \\
&= \int_{\mathbb{H}^{n}} P(d(x, \xi), y) f(\xi) \,\mathrm{d}\xi
\end{split}
\end{equation*}
in the even-dimensional case as well, where we used \Cref{lem:recurrence} in the last equality.

It only remains to prove that $V$ solves the extension problem \eqref{eq:ext-prob}. Since the heat semigroup $e^{t\Delta_{\mathbb{H}^{n}}}[f]$ solves \eqref{eq:Cauchy}, $V$ satisfies
\begin{equation*}
\Delta_{x} V = \frac{y^{sp}}{2^{sp} \Gamma(\frac{sp}{2})} \int_{0}^{\infty} \partial_{t} \left( e^{t\Delta_{\mathbb{H}^{n}}}[f](x) \right) e^{-\frac{y^2}{4t}} t^{-1-\frac{sp}{2}} \,\mathrm{d}t.
\end{equation*}
Using the integration by parts and the fact that $|e^{t\Delta_{\mathbb{H}^{n}}}[f](x)| \leq \|f\|_{L^{\infty}}$, we obtain
\begin{equation*}
\begin{split}
\Delta_{x} V
&= \frac{y^{sp}}{2^{sp} \Gamma(\frac{sp}{2})} \left( \left[ e^{t\Delta_{\mathbb{H}^{n}}}[f](x) e^{-\frac{y^{2}}{4t}}t^{-1-\frac{sp}{2}} \right]_{0}^{\infty} - \int_{0}^{\infty} e^{t\Delta_{\mathbb{H}^{n}}}[f](x) \partial_{t} \left( e^{-\frac{y^{2}}{4t}} t^{-1-\frac{sp}{2}} \right) \mathrm{d}t \right) \\
&= - \frac{y^{sp}}{2^{sp} \Gamma(\frac{sp}{2})} \int_{0}^{\infty} e^{t\Delta_{\mathbb{H}^{n}}}[f](x) \left(\frac{y^{2}}{4} e^{-\frac{y^{2}}{4t}} t^{-3-\frac{sp}{2}} - \left(1+\frac{sp}{2} \right) e^{-\frac{y^{2}}{4t}} t^{-2-\frac{sp}{2}} \right) \mathrm{d}t.
\end{split}
\end{equation*}
Since
\begin{equation*}
\begin{split}
V_{y}
&= \frac{sp y^{sp-1}}{2^{sp} \Gamma(\frac{sp}{2})} \int_{0}^{\infty} e^{t\Delta_{\mathbb{H}^{n}}}[f](x) e^{-\frac{y^{2}}{4t}} t^{-1-\frac{sp}{2}} \,\mathrm{d}t \\
&\quad - \frac{y^{sp+1}}{2^{sp+1} \Gamma(\frac{sp}{2})} \int_{0}^{\infty} e^{t\Delta_{\mathbb{H}^{n}}}[f](x) e^{-\frac{y^{2}}{4t}} t^{-2-\frac{sp}{2}} \,\mathrm{d}t
\end{split}
\end{equation*}
and
\begin{equation*}
\begin{split}
V_{yy}
&= \frac{sp(sp-1) y^{sp-2}}{2^{sp} \Gamma(\frac{sp}{2})} \int_{0}^{\infty} e^{t\Delta_{\mathbb{H}^{n}}}[f](x) e^{-\frac{y^{2}}{4t}} t^{-1-\frac{sp}{2}} \,\mathrm{d}t \\
&\quad - \frac{2sp+1}{2^{sp+1}\Gamma(\frac{sp}{2})} y^{sp} \int_{0}^{\infty} e^{t\Delta_{\mathbb{H}^{n}}}[f](x) e^{-\frac{y^{2}}{4t}} t^{-2-\frac{sp}{2}} \,\mathrm{d}t \\
&\quad + \frac{y^{sp+2}}{2^{sp+2}\Gamma(\frac{sp}{2})} \int_{0}^{\infty} e^{t\Delta_{\mathbb{H}^{n}}}[f](x) e^{-\frac{y^{2}}{4t}} t^{-3-\frac{sp}{2}} \,\mathrm{d}t,
\end{split}
\end{equation*}
one can easily compute
\begin{equation*}
\Delta_{x} V(x,y) + \frac{1-sp}{y} V_{y}(x,y) + V_{yy}(x,y) = 0.
\end{equation*}
Finally, we prove $V(x,0) = f(x)$. Indeed, since the heat kernel $h(t, \rho)$ satisfies
\begin{equation*}
\int_{\mathbb{H}^{n}} h(t, d(x, \xi)) \,\mathrm{d}\xi = 1,
\end{equation*}
we obtain
\begin{equation*}
\begin{split}
\int_{\mathbb{H}^{n}} P(d(x,\xi), y) \,\mathrm{d}\xi
&= \frac{y^{sp}}{2^{sp}\Gamma(\frac{sp}{2})} \int_{0}^{\infty} \left( \int_{\mathbb{H}^{n}} h(t, d(x,\xi)) \,\mathrm{d}\xi \right) e^{-\frac{y^{2}}{4t}} \frac{\mathrm{d}t}{t^{1+\frac{sp}{2}}} \\
&= \frac{y^{sp}}{2^{sp}\Gamma(\frac{sp}{2})} \int_{0}^{\infty} e^{-\frac{y^{2}}{4t}} \frac{\mathrm{d}t}{t^{1+\frac{sp}{2}}} =1.
\end{split}
\end{equation*}
Now, fix $\varepsilon>0$ and find $\delta>0$ such that $|f(\xi)-f(x)|<\varepsilon$ whenever $d(x, \xi) < \delta$. Then
\begin{align*}
|V(x, y)-f(x)|
&\leq \int_{d(x, \xi)<\delta} P(d(x, \xi), y) |f(\xi)-f(x)| \,\mathrm{d}\xi + 2\|f\|_{L^\infty} \int_{d(x, \xi)<\delta} P(d(x, \xi), y) \,\mathrm{d}\xi \\
&\leq \varepsilon + 2\|f\|_{L^\infty} \int_{d(x, \xi)<\delta} P(d(x, \xi), y) \,\mathrm{d}\xi.
\end{align*}
Since $P(\rho, y) \to 0$ uniformly for $\rho \geq \delta$ as $y \searrow 0$, we obtain
\begin{equation*}
\limsup_{y \searrow 0} |V(x, y)-f(x)| \leq \varepsilon.
\end{equation*}
Since $\varepsilon>0$ was arbitrary, this shows that $V(x, 0)=f(x)$. This concludes that $V$ solves the extension problem \eqref{eq:ext-prob}.
\end{proof}

Let us now prove \Cref{thm:extension} by using the Poisson formula in \Cref{lem:Poisson}.

\begin{proof} [Proof of \Cref{thm:extension}]
Since $c_{n, s, p}C_{1} = C_{3}C_{4}$, by \Cref{lem:Poisson} it is enough to show
\begin{equation*}
\left| \, \mathrm{P.V.}\int_{\mathbb{H}^{n}} \Phi_{p}(u(x)-u(\xi)) K(d(x, \xi)) \,\mathrm{d}\xi \right| \to 0
\end{equation*}
as $y \searrow 0$, where
\begin{equation*}
K(\rho) = \left( \frac{-\partial_{\rho}}{\sinh \rho} \right)^{\frac{n-1}{2}} \left( \widetilde{K}_{\frac{1+sp}{2}, \frac{n-1}{2}}(\rho) - \widetilde{K}_{\frac{1+sp}{2}, \frac{n-1}{2}}\left( \sqrt{\rho^{2}+y^{2}} \right) \right)
\end{equation*}
when $n$ is odd and
\begin{equation*}
K(\rho) = \int_{\rho}^{\infty} \frac{\sinh r}{\sqrt{\pi}\sqrt{\cosh r - \cosh \rho}} \left( \frac{-\partial_{r}}{\sinh r} \right)^{\frac{n}{2}} \left( \widetilde{K}_{\frac{1+sp}{2},\frac{n-1}{2}}(r) - \widetilde{K}_{\frac{1+sp}{2}, \frac{n-1}{2}}\left( \sqrt{r^{2}+y^{2}} \right) \right) \mathrm{d}r
\end{equation*}
when $n$ is even.

We first split the integral as follows:
\begin{equation*}
\begin{split}
&\left| \int_{\mathbb{H}^{n}} \Phi_{p}(u(x)-u(\xi)) K(d(x, \xi)) \,\mathrm{d}\xi \right| \\
&\leq \left| \int_{d(x,\xi) \leq 1} \Phi_{p}(u(x)-u(\xi)) K(d(x, \xi)) \,\mathrm{d}\xi \right| + \left| \int_{d(x,\xi) > 1} \Phi_p(u(x)-u(\xi)) K(d(x, \xi)) \,\mathrm{d}\xi \right| \\
&= J_{1}+J_{2}.
\end{split}
\end{equation*}
For $J_{1}$, we apply \Cref{lem:appendix} to $K\chi_{\{d(x,\xi)\leq 1\}}$ to obtain
\begin{equation*}
J_{1} \lesssim \int_{d(x, \xi) \leq 1} d(x, \xi)^{\alpha} |K(d(x, \xi))| \,\mathrm{d}\xi,
\end{equation*}
where $\alpha = 2p-2$ when $p \in (\frac{2}{2-s},2)$ and $\alpha = p$ when $p \in (1, \frac{2}{2-s}] \cup [2, \infty)$. For $J_{2}$, we have
\begin{equation*}
J_{2} \lesssim \|u\|_{L^{\infty}(\mathbb{H}^{n})}^{p-1} \int_{d(x, \xi) > 1} |K(d(x, \xi))| \,\mathrm{d}\xi.
\end{equation*}
Therefore, the dominated convergence theorem concludes that $J_{1} + J_{2} \to 0$ as $y \searrow 0$.
\end{proof}


\section{Pointwise convergence} \label{sec:convergence}


This section is devoted to the proof of \Cref{thm:convergence}. As mentioned in \Cref{sec:introduction}, the limits of the integrals
\begin{equation} \label{eq:integral1}
c_{n,s,p} \int_{R}^{\infty} \mathcal{K}_{n,s,p}(\rho)\sinh^{n-1}\rho \,\mathrm{d}\rho, \quad c_{n,s,p} \int_{0}^{R} \rho^{p} \mathcal{K}_{n,s,p}(\rho)\sinh^{n-1}\rho \,\mathrm{d}\rho,
\end{equation}
and
\begin{equation} \label{eq:integral2}
c_{n,s,p} \int_{0}^{R} \rho^{\beta+p} \mathcal{K}_{n,s,p} \sinh^{n-1}\rho \,\mathrm{d}\rho, \quad \beta > 0,
\end{equation}
as $s \to 1^{-}$, play a key role in the proof of \Cref{thm:convergence}. Here, we recall that the constant $c_{n, s, p}$ is given in \eqref{eq-cnsp}.

In the following series of lemmas, we compute limits of the integrals in \eqref{eq:integral1} and \eqref{eq:integral2}.

\begin{lemma} \label{lem:infty}
Let $n\geq 2$ and $p > 1$. For any $R > 0$,
\begin{equation*}
\lim_{s \nearrow 1} c_{n, s, p} \int_{R}^{\infty} \mathcal{K}_{n,s,p}(\rho)\sinh^{n-1} \rho \,\mathrm{d}\rho = 0.
\end{equation*}
\end{lemma}

\begin{proof}
Let us first consider the case $n=2m+1$ with $m \geq 1$. Since $\lim_{s \nearrow 1} (1-s)|\Gamma(-s)| = 1$, we have $c_{n, s, p}C_{1} \leq C(1-s)$ for some $C=C(n, p) > 0$. By using \Cref{cor:nonnegativity}, we have
\begin{equation} \label{eq:infty}
\begin{split}
0
&\leq c_{n, s, p} \int_{R}^{\infty} \mathcal{K}_{n,s,p}(\rho)\sinh^{n-1} \rho \,\mathrm{d}\rho \\
&\lesssim (1-s) \int_{R}^{\infty} \sinh^{2m} \rho \left( \frac{-\partial_{\rho}}{\sinh \rho} \right)^{m} \widetilde{K}_{\frac{1+sp}{2},m}(\rho) \, \mathrm{d}\rho.
\end{split}
\end{equation}
Thus, it is enough to show that the right-hand side of \eqref{eq:infty} converges to zero as $s \to 1^{-}$. We actually prove the following stronger statement:
\begin{equation} \label{eq:infty-odd}
\lim_{s \nearrow 1} (1-s) \int_{R}^{\infty} \sinh^{m+a}\rho \left( \frac{-\partial_{\rho}}{\sinh \rho} \right)^{m} \widetilde{K}_{\frac{1+sp}{2},a}(\rho) \, \mathrm{d}\rho = 0 \quad\text{for each } a > 0.
\end{equation}

We use the induction on $m$. When $m=1$, using \eqref{eq:scrK-recurrence} and the fact that $K_{\nu}$ is increasing with respect to $\nu > 0$, we have
\begin{equation*}
\begin{split}
\int_{R}^{\infty} \sinh^{1+a}\rho \left( \frac{-\partial_{\rho}}{\sinh \rho} \right) \widetilde{K}_{\frac{1+sp}{2},a}(\rho)\, \mathrm{d}\rho
&= a \int_{R}^{\infty} (\sinh^{a} \rho) \rho^{-\frac{1+sp}{2}} K_{\frac{3+sp}{2}}(a\rho) \, \mathrm{d}\rho \\
&\leq a \int_{R}^{\infty} (\sinh^{a} \rho) \rho^{-\frac{1+sp}{2}} K_{\frac{3+p}{2}}(a\rho) \, \mathrm{d}\rho.
\end{split}
\end{equation*}
By \eqref{eq:K-asymptotic}, there exists $M = M(p) > 1$ such that
\begin{equation} \label{eq:M}
K_{\frac{3+p}{2}}(\rho) \leq \sqrt{\frac{\pi}{\rho}} e^{-\rho}\quad\text{for } \rho > M.
\end{equation}
The inequalities $\rho^{-\frac{1+sp}{2}} \leq \max\lbrace \rho^{-\frac{1}{2}}, \rho^{-\frac{1+p}{2}} \rbrace$ and $\sinh \rho < e^{\rho}$, together with \eqref{eq:M}, yield
\begin{equation*}
\begin{split}
&\int_{R}^{\infty} (\sinh^{a} \rho) \rho^{-\frac{1+sp}{2}} K_{\frac{3+p}{2}}(a\rho) \, \mathrm{d}\rho \\
&\leq \int_{R}^{M/a} (\sinh^{a} \rho) \max\left\lbrace \rho^{-\frac{1}{2}}, \rho^{-\frac{1+p}{2}} \right\rbrace K_{\frac{3+p}{2}}(a\rho) \, \mathrm{d}\rho + \sqrt{\frac{\pi}{a}} \int_{M/a}^{\infty} \rho^{-1-\frac{sp}{2}} \, \mathrm{d}\rho.
\end{split}
\end{equation*}
Note that the first integral in the right-hand side of the inequality above is a constant depending on $a$, $p$, and $R$ only. For the second integral, we estimate
\begin{equation*}
\sqrt{\frac{\pi}{a}} \int_{M/a}^{\infty} \rho^{-1-\frac{sp}{2}} \, \mathrm{d}\rho = \frac{2}{sp} \sqrt{\frac{\pi}{a}} \left( \frac{a}{M} \right)^{\frac{sp}{2}} \leq \frac{2}{sp} \sqrt{\frac{\pi}{a}} \max \left\lbrace \left( \frac{a}{M} \right)^{\frac{p}{2}}, 1 \right\rbrace.
\end{equation*}
Thus, we arrive at
\begin{equation*}
\lim_{s \nearrow 1} (1-s) \int_{R}^{\infty} \sinh^{1+a}\rho \left( \frac{-\partial_{\rho}}{\sinh \rho} \right) \widetilde{K}_{\frac{1+sp}{2},a}(\rho)\, \mathrm{d}\rho = 0,
\end{equation*}
which proves \eqref{eq:infty-odd} for $m=1$.

Assume now that \eqref{eq:infty-odd} is true for $m$ and prove it for $m+1$. Using integration by parts, we have
\begin{align*}
&\int_{R}^{\infty} \sinh^{m+1+a}\rho \left( \frac{-\partial_{\rho}}{\sinh \rho} \right)^{m+1} \widetilde{K}_{\frac{1+sp}{2},a}(\rho) \, \mathrm{d}\rho \\
&= (m+a) \int_{R}^{\infty} \sinh^{m+a-1}\rho \cosh \rho \left( \frac{-\partial_{\rho}}{\sinh \rho} \right)^{m} \widetilde{K}_{\frac{1+sp}{2},a}(\rho) \, \mathrm{d}\rho \\
&\quad - \left[ \sinh^{m+a}\rho \left( \frac{-\partial_\rho}{\sinh \rho} \right)^m \widetilde{K}_{\frac{1+sp}{2},a}(\rho) \right]_R^{\infty}.
\end{align*}
Note that by \Cref{lem-asymptotic} and \eqref{eq:K-asymptotic},
\begin{equation*}
\sinh^{m+a}\rho \left( \frac{-\partial_\rho}{\sinh \rho} \right)^m \widetilde{K}_{\frac{1+sp}{2},a}(\rho) \sim \rho^{-\frac{1+sp}{2}} (\sinh\rho)^{a} K_{m+\frac{1+sp}{2}}(a\rho) \sim \rho^{-1-sp/2}
\end{equation*}
as $\rho \to \infty$. Thus, taking the limit $\lim_{s \nearrow 1}(1-s)$ yields
\begin{equation*}
\begin{split}
&\lim_{s \nearrow 1} (1-s) \int_{R}^{\infty} \sinh^{m+1+a}\rho \left( \frac{-\partial_{\rho}}{\sinh \rho} \right)^{m+1} \widetilde{K}_{\frac{1+sp}{2},a}(\rho) \, \mathrm{d}\rho \\
&= \lim_{s \nearrow 1} (1-s) (m+a) \int_{R}^{\infty} \sinh^{m+a-1}\rho \cosh \rho \left( \frac{-\partial_{\rho}}{\sinh \rho} \right)^{m} \widetilde{K}_{\frac{1+sp}{2},a}(\rho) \, \mathrm{d}\rho.
\end{split}
\end{equation*}
Therefore, by an inequality
\begin{equation} \label{eq:coth}
\cosh \rho \leq \coth R \sinh \rho \quad\text{for } \rho \geq R,
\end{equation}
\Cref{lem:positivity}, and the induction hypothesis, we conclude
\begin{equation*}
\begin{split}
&\lim_{s \nearrow 1} (1-s) \int_{R}^{\infty} \sinh^{m+1+a}\rho \left( \frac{-\partial_{\rho}}{\sinh \rho} \right)^{m+1} \widetilde{K}_{\frac{1+sp}{2},a}(\rho) \, \mathrm{d}\rho \\
&\leq (m+a) (\coth R) \lim_{s \nearrow 1} (1-s) \int_{R}^{\infty} \sinh^{m+a}\rho \left( \frac{-\partial_{\rho}}{\sinh \rho} \right)^{m} \widetilde{K}_{\frac{1+sp}{2},a}(\rho) \, \mathrm{d}\rho = 0.
\end{split}
\end{equation*}
This finishes the proof of the lemma in the odd-dimensional case.

Let us next consider the even-dimensional cases $n=2m$ with $m \geq 1$. Similarly as in the odd-dimensional case, since
\begin{equation*}
\begin{split}
0
&\leq c_{n, s, p} \int_{R}^{\infty} \mathcal{K}_{n,s,p}(\rho)\sinh^{n-1} \rho \,\mathrm{d}\rho \\
&\lesssim (1-s) \int_{R}^{\infty} \sinh^{2m-1} \rho \int_{\rho}^{\infty} \frac{\sinh r}{\sqrt{\cosh r - \cosh \rho}} \left( \frac{-\partial_{r}}{\sinh r} \right)^{m} \widetilde{K}_{\frac{1+sp}{2}, \frac{2m-1}{2}}(r) \,\mathrm{d}r \,\mathrm{d}\rho,
\end{split}
\end{equation*}
the desired result will follow once we prove the following:
\begin{equation} \label{eq:infty-even}
\begin{split}
\lim_{s \nearrow 1} (1-s) \int_{R}^{\infty} \sinh^{\frac{2m-1}{2}+a}\rho \int_{\rho}^{\infty} \frac{\sinh r}{\sqrt{\cosh r - \cosh \rho}} \left( \frac{-\partial_r}{\sinh r} \right)^{m} \widetilde{K}_{\frac{1+sp}{2},a}(r) \, \mathrm{d}r \,\mathrm{d}\rho = 0 \\
\quad\text{for each } a \geq 1/2.
\end{split}
\end{equation}
If $m=1$, then
\begin{equation*}
\begin{split}
&\int_{R}^{\infty} \sinh^{\frac{1}{2}+a}\rho \int_{\rho}^{\infty} \frac{\sinh r}{\sqrt{\cosh r - \cosh \rho}} \left( \frac{-\partial_{r}}{\sinh r} \right) \widetilde{K}_{\frac{1+sp}{2},a}(r) \,\mathrm{d}r \,\mathrm{d}\rho \\
&\leq a \int_{R}^{M/a} \sinh^{\frac{1}{2}+a}\rho \int_{\rho}^{\infty} \frac{1}{\sqrt{\cosh r - \cosh \rho}} r^{-\frac{1+sp}{2}} K_{\frac{3+p}{2}}(ar) \,\mathrm{d}r \,\mathrm{d}\rho \\
&\quad + a \int_{M/a}^{\infty} \sinh^{\frac{1}{2}+a}\rho \int_{\rho}^{\infty} \frac{1}{\sqrt{\cosh r - \cosh \rho}} r^{-\frac{1+sp}{2}} K_{\frac{3+p}{2}}(ar) \,\mathrm{d}r \,\mathrm{d}\rho =: J_{1} + J_{2}.
\end{split}
\end{equation*}
For $J_{2}$, we use \eqref{eq:M} to obtain
\begin{equation*}
J_{2} \leq \sqrt{\pi a} \int_{M/a}^{\infty} \frac{\sinh^{\frac{1}{2}+a}\rho}{\rho^{1+\frac{sp}{2}} e^{a\rho}} \int_{\rho}^{\infty} \frac{1}{\sqrt{\cosh r - \cosh \rho}} \,\mathrm{d}r \,\mathrm{d}\rho.
\end{equation*}
Since
\begin{equation*}
\begin{split}
\int_{\rho}^{\infty} \frac{1}{\sqrt{\cosh r - \cosh \rho}} \,\mathrm{d}r 
&= \frac{1}{\sqrt{2}} \int_{\rho}^{\infty} \frac{1}{\sqrt{\sinh\frac{r+\rho}{2} \sinh \frac{r-\rho}{2}}} \,\mathrm{d}r \\
&\leq \frac{1}{\sqrt{2\sinh \rho}} \int_{\rho}^{\infty} \frac{1}{\sqrt{\sinh\frac{r-\rho}{2}}} \,\mathrm{d}r \\
&= \frac{1}{\sqrt{2\sinh \rho}} \int_{0}^{\infty} \frac{1}{\sqrt{\sinh\frac{r}{2}}} \,\mathrm{d}r = \frac{\Gamma(1/4)}{\Gamma(3/4)}\sqrt{\frac{\pi}{\sinh \rho}}
\end{split}
\end{equation*}
and $\sinh^{a} \rho \leq e^{a\rho}$, we have
\begin{equation*}
J_{2} \leq \frac{\Gamma(1/4)}{\Gamma(3/4)} \pi \sqrt{a} \int_{M/a}^{\infty} \rho^{-1-\frac{sp}{2}} \,\mathrm{d}\rho = \frac{\Gamma(1/4)}{\Gamma(3/4)} \frac{\pi \sqrt{a}}{sp} \left( \frac{a}{M} \right)^{\frac{sp}{2}}.
\end{equation*}
On the other hand, for $J_{1}$ we observe
\begin{equation*}
J_{1} \leq a \int_{R}^{M/a} \sinh^{\frac{1}{2}+a}\rho \int_{\rho}^{\infty} \frac{\max \lbrace r^{-\frac{1+p}{2}}, r^{-\frac{1}{2}} \rbrace}{\sqrt{\cosh r - \cosh \rho}} K_{\frac{3+p}{2}}(ar) \,\mathrm{d}r \,\mathrm{d}\rho.
\end{equation*}
Since the inner integral is continuous and integrable on $[R, M/a]$, $J_{1}$ is controlled by some constant $C = C(a, p, R) > 0$. Therefore, we conclude $\lim_{s \nearrow 1} (1-s) (J_{1} + J_{2}) = 0$, which proves \eqref{eq:infty-even} for $m=1$.

Finally, let us assume that \eqref{eq:infty-even} holds for $m$ and prove it for $m+1$. By \Cref{lem:recurrence}, we have
\begin{equation*}
\begin{split}
&\int_{R}^{\infty} \sinh^{\frac{2m+1}{2}+a}\rho \int_{\rho}^{\infty} \frac{\sinh r}{\sqrt{\cosh r - \cosh \rho}} \left( \frac{-\partial_{r}}{\sinh r} \right)^{m+1} \widetilde{K}_{\frac{1+sp}{2},a}(r) \,\mathrm{d}r \,\mathrm{d}\rho \\
&= \int_{R}^{\infty} \sinh^{\frac{2m+1}{2}+a}\rho \left( \frac{-\partial_{\rho}}{\sinh \rho} \right) \int_{\rho}^{\infty} \frac{\sinh r}{\sqrt{\cosh r - \cosh \rho}} \left( \frac{-\partial_{r}}{\sinh r} \right)^{m} \widetilde{K}_{\frac{1+sp}{2},a}(r) \,\mathrm{d}r \,\mathrm{d}\rho.
\end{split}
\end{equation*}
Using integration by parts, \eqref{eq:coth}, and \Cref{lem:positivity} and then taking $\lim_{s \nearrow 1}(1-s)$ as in the odd-dimensional case, we deduce
\begin{equation*}
\begin{split}
&\lim_{s \nearrow 1} (1-s) \int_{R}^{\infty} \sinh^{\frac{2m+1}{2}+a}\rho \int_{\rho}^{\infty} \frac{\sinh r}{\sqrt{\cosh r - \cosh \rho}} \left( \frac{-\partial_{r}}{\sinh r} \right)^{m+1} \widetilde{K}_{\frac{1+sp}{2},a}(r) \,\mathrm{d}r \,\mathrm{d}\rho \\
&\leq C \lim_{s \nearrow 1} (1-s) \int_{R}^{\infty} \sinh^{\frac{2m-1}{2}+a}\rho \int_{\rho}^{\infty} \frac{\sinh r}{\sqrt{\cosh r - \cosh \rho}} \left( \frac{-\partial_{r}}{\sinh r} \right)^{m} \widetilde{K}_{\frac{1+sp}{2},a}(r) \,\mathrm{d}r \,\mathrm{d}\rho
\end{split}
\end{equation*}
for some $C = C(m, a, R)$. Therefore, the statement \eqref{eq:infty-even} for $m+1$ follows by the induction hypothesis.
\end{proof}

\begin{lemma} \label{lem:zero}
Let $n\geq 2$ and $p > 1$. For any $R > 0$,
\begin{equation} \label{eq:zero}
\lim_{s \nearrow 1} c_{n, s, p} \int_{0}^{R} \rho^{p} \mathcal{K}_{n,s,p}(\rho)\sinh^{n-1}\rho \,\mathrm{d}\rho = \frac{1}{\pi^{\frac{n-1}{2}}} \frac{\Gamma(\frac{p+n}{2})}{\Gamma(\frac{p+1}{2})}.
\end{equation}
\end{lemma}

\begin{proof}
Let us first consider the odd-dimensional case $n=2m+1$ with $m \geq 1$. One can easily check that \eqref{eq:zero} is equivalent to
\begin{equation} \label{eq:odd}
\begin{split}
&\lim_{s \nearrow 1} (1-s) \int_{0}^{R} \rho^{p} \sinh^{2m} \rho \left( \frac{-\partial_{\rho}}{\sinh \rho} \right)^{m} \widetilde{K}_{\frac{1+sp}{2},m}(\rho) \, \mathrm{d}\rho \\
&= \frac{2^{m-1}}{p} \left( \frac{2}{m} \right)^{\frac{p+1}{2}} \Gamma\left( \frac{p+2m+1}{2} \right)
\end{split}
\end{equation}
by using $\lim_{s \nearrow 1} (1-s)|\Gamma(-s)| = 1$. Actually, we will prove the following statement, which is slightly stronger than \eqref{eq:odd}:
\begin{equation} \label{eq:zero-odd}
\begin{split}
\lim_{s \nearrow 1} (1-s) \int_{0}^{R} \rho^{p} \sinh^{2m} \rho \left( \frac{-\partial_{\rho}}{\sinh \rho} \right)^{m} \widetilde{K}_{\frac{1+sp}{2},a}(\rho) \, \mathrm{d}\rho \\
= \frac{2^{m-1}}{p} \left( \frac{2}{a} \right)^{\frac{p+1}{2}} \Gamma\left( \frac{p+2m+1}{2} \right) \quad\text{for each } a \geq 1.
\end{split}
\end{equation}
Let $\varepsilon \in (0,1)$, then there exists $\delta_{0} \in (0,1)$ such that
\begin{equation} \label{eq:delta-sinh}
1-\varepsilon \leq \frac{\sinh \rho}{\rho} \leq 1+\varepsilon
\end{equation}
for all $\rho \in (0, \delta_{0})$. Moreover, using the asymptotic behavior \eqref{eq:K-asymptotic} of the modified Bessel function, for each $s \in [0,1]$ we find $\delta_{s} > 0$ such that
\begin{equation} \label{eq:delta-K}
\frac{1-\varepsilon}{2} \Gamma\left( \frac{3+sp}{2} \right)\left( \frac{\rho}{2} \right)^{-\frac{3+sp}{2}} \leq K_{\frac{3+sp}{2}}(\rho) \leq \frac{1+\varepsilon}{2} \Gamma\left( \frac{3+sp}{2} \right)\left( \frac{\rho}{2} \right)^{-\frac{3+sp}{2}}
\end{equation}
for all $\rho \in (0, \delta_{s})$. Furthermore, since $\{K_{\nu}\}_{\nu \in [3/2, (3+p)/2]}$ is equicontinuous, we may assume that $\delta_{s}>0$ has been chosen continuously on $s$. Let us take $\delta = \delta_{0} \land \min_{s \in [0,1]} \delta_{s} \land R$, then $\delta = \delta(\varepsilon, p, R) > 0$, and \eqref{eq:delta-sinh} and \eqref{eq:delta-K} hold for all $\rho \in (0, \delta)$.

We fix $a \geq 1$ and denote by $G_{s, p, m, a}(\rho)$ the integrand in the left-hand side of \eqref{eq:zero-odd}. Then, $|G_{s, p, m, a}(\rho)|$ is bounded by the function $\sup_{0 \leq s \leq 1} |G_{s, p, m, a}(\rho)|$, which is independent of $s$ and bounded on a compact interval $[\delta/a, R]$. Thus, we have
\begin{equation*}
\lim_{s \nearrow 1} (1-s) \int_{\delta/a}^{R} G_{s, p, m, a}(\rho) \, \mathrm{d}\rho = 0,
\end{equation*}
and hence
\begin{equation*}
\lim_{s \nearrow 1} (1-s) \int_{0}^{R} G_{s, p, m, a}(\rho) \, \mathrm{d}\rho = \lim_{s \nearrow 1} (1-s) \int_{0}^{\delta/a} G_{s, p, m, a}(\rho) \, \mathrm{d}\rho.
\end{equation*}
Let us now prove \eqref{eq:zero-odd} by induction. When $m=1$, we first use \eqref{eq:scrK-recurrence} to have
\begin{equation*}
G_{s, p, 1, a}(\rho) = a\rho^{p-\frac{1+sp}{2}} K_{\frac{3+sp}{2}}(a\rho) \sinh \rho.
\end{equation*}
If $\rho < \delta/a$, then $\rho \leq a\rho < \delta$ since $a \geq 1$. Thus, we utilize \eqref{eq:delta-sinh} and \eqref{eq:delta-K} to obtain
\begin{equation*}
(1-\varepsilon)^2 \left( \frac{2}{a} \right)^{\frac{1+sp}{2}} \Gamma\left( \frac{3+sp}{2} \right) \rho^{p(1-s)-1} \leq G_{s, p, 1, a}(\rho) \leq (1+\varepsilon)^2 \left( \frac{2}{a} \right)^{\frac{1+sp}{2}} \Gamma\left( \frac{3+sp}{2} \right) \rho^{p(1-s)-1}.
\end{equation*}
This leads us to the inequalities
\begin{equation*}
\begin{split}
\lim_{s \nearrow 1} (1-s) \int_{0}^{R} G_{s, p, 1, a}(\rho) \, \mathrm{d}\rho
&= \lim_{s \nearrow 1} (1-s) \int_{0}^{\delta/a} G_{s, p, 1, a}(\rho) \, \mathrm{d}\rho \\
&\leq \lim_{s \nearrow 1} (1-s) (1+\varepsilon)^2 \left( \frac{2}{a} \right)^{\frac{1+sp}{2}} \Gamma\left( \frac{3+sp}{2} \right) \int_{0}^{\delta/a} \rho^{p(1-s)-1} \, \mathrm{d}\rho \\
&= (1+\varepsilon)^2 \frac{1}{p} \left( \frac{2}{a} \right)^{\frac{p+1}{2}} \Gamma\left( \frac{p+3}{2} \right)
\end{split}
\end{equation*}
and
\begin{equation*}
\lim_{s \nearrow 1} (1-s) \int_{0}^{R} G_{s, p, 1, a}(\rho) \, \mathrm{d}\rho \geq (1-\varepsilon)^{2} \frac{1}{p} \left( \frac{2}{a} \right)^{\frac{p+1}{2}} \Gamma\left( \frac{p+3}{2} \right).
\end{equation*}
Therefore, the statement \eqref{eq:zero-odd} for $m=1$ follows by taking $\varepsilon \to 0$.

Assume now that \eqref{eq:zero-odd} holds for $m \geq 1$. Then, a similar argument shows
\begin{equation*}
\begin{split}
&\lim_{s \nearrow 1} (1-s) \int_{0}^{R} G_{s, p, m+1, a}(\rho) \, \mathrm{d}\rho \\
&= \lim_{s \nearrow 1} (1-s) \int_{0}^{\delta/a} \rho^{p} \sinh^{2m+2} \rho \left( \frac{-\partial_{\rho}}{\sinh \rho} \right)^{m+1} \widetilde{K}_{\frac{1+sp}{2},a}(\rho) \, \mathrm{d}\rho \\
&\leq \lim_{s \nearrow 1} (1-s)(1+\varepsilon)^{2m+1} \int_{0}^{\delta/a} \rho^{p+2m+1}(-\partial_\rho) \left( \frac{-\partial_{\rho}}{\sinh \rho} \right)^{m} \widetilde{K}_{\frac{1+sp}{2},a}(\rho) \, \mathrm{d}\rho,
\end{split}
\end{equation*}
where nonnegativity of the integrands follows from \Cref{lem:positivity}. Using the integration by parts, \eqref{eq:delta-sinh}, and the induction hypothesis, we arrive at
\begin{equation*}
\begin{split}
&\lim_{s \nearrow 1} (1-s) \int_{0}^{R} G_{s, p, m+1, a}(\rho) \, \mathrm{d}\rho \\
&\leq (1+\varepsilon)^{2m+1} (p+2m+1) \lim_{s \nearrow 1} (1-s) \int_{0}^{\delta/a} \rho^{p+2m} \left( \frac{-\partial_\rho}{\sinh \rho} \right)^{m} \widetilde{K}_{\frac{1+sp}{2},a}(\rho) \, \mathrm{d}\rho \\
&\leq \frac{(1+\varepsilon)^{2m+1}}{(1-\varepsilon)^{2m}} (p+2m+1) \lim_{s \nearrow 1} (1-s) \int_{0}^{\delta/a} \rho^{p} \sinh^{2m}\rho \left( \frac{-\partial_\rho}{\sinh \rho} \right)^{m} \widetilde{K}_{\frac{1+sp}{2},a}(\rho) \, \mathrm{d}\rho \\
&= \frac{(1+\varepsilon)^{2m+1}}{(1-\varepsilon)^{2m}} (p+2m+1) \frac{2^{m-1}}{p} \left( \frac{2}{a} \right)^{\frac{p+1}{2}} \Gamma\left( \frac{p+2m+1}{2} \right) \\
&= \frac{(1+\varepsilon)^{2m+1}}{(1-\varepsilon)^{2m}} \frac{2^{m}}{p} \left( \frac{2}{a} \right)^{\frac{p+1}{2}} \Gamma\left( \frac{p+2m+3}{2} \right).
\end{split}
\end{equation*}
Similarly, we obtain
\begin{equation*}
\lim_{s \nearrow 1} (1-s) \int_{0}^{R} G_{s, p, m+1, a}(\rho) \, \mathrm{d}\rho \geq \frac{(1-\varepsilon)^{2m+1}}{(1+\varepsilon)^{2m}} \frac{2^{m}}{p} \left( \frac{2}{a} \right)^{\frac{p+1}{2}} \Gamma\left( \frac{p+2m+3}{2} \right),
\end{equation*}
from which \eqref{eq:zero-odd} for $m+1$ follows by taking $\varepsilon \to 0$. The statement \eqref{eq:zero-odd} has been proved for all $m \in \mathbb{N}$, finishing the proof of \eqref{eq:zero} for the odd-dimensional case.

Let us next consider the even-dimensional case $n = 2m$ with $m \geq 1$. In this case, \eqref{eq:zero} is equivalent to
\begin{equation*}
\begin{split}
&\lim_{s \nearrow 1} (1-s) \int_{0}^{R} \rho^{p} \sinh^{2m-1}\rho \int_{\rho}^{\infty} \frac{\sinh r}{\sqrt{\cosh r - \cosh \rho}} \left( \frac{-\partial_{r}}{\sinh r} \right)^{m} \widetilde{K}_{\frac{1+sp}{2},\frac{2m-1}{2}}(r) \, \mathrm{d}r \,\mathrm{d}\rho \\
&= \sqrt{\frac{\pi}{2}} \frac{2^{m-1}}{p} \left( \frac{2}{m-1/2} \right)^{\frac{p+1}{2}} \Gamma\left(\frac{p+2m}{2} \right).
\end{split}
\end{equation*}
As in the odd-dimensional case, we will prove a stronger statement:
\begin{equation} \label{eq:zero-even}
\begin{split}
&\lim_{s \nearrow 1} (1-s) \int_{0}^{R} \rho^{p} \sinh^{2m-1}\rho \int_{\rho}^{\infty} \frac{\sinh r}{\sqrt{\cosh r - \cosh \rho}} \left( \frac{-\partial_{r}}{\sinh r} \right)^{m} \widetilde{K}_{\frac{1+sp}{2},a}(r) \, \mathrm{d}r \,\mathrm{d}\rho \\
&= \sqrt{\frac{\pi}{2}} \frac{2^{m-1}}{p} \left( \frac{2}{a} \right)^{\frac{p+1}{2}} \Gamma\left(\frac{p+2m}{2} \right) \quad\text{for each } a \geq 1/2.
\end{split}
\end{equation}
Recall that we have taken $\delta$ so that \eqref{eq:delta-sinh} and \eqref{eq:delta-K} hold for all $\rho \in (0, \delta)$. Let us fix $a \geq 1/2$. By \Cref{lem:even_dim}, for each $s \in [0,1]$ we find $\tilde{\delta}_{s} > 0$ such that
\begin{equation} \label{eq:delta-intK}
\begin{split}
(1-\varepsilon) \sqrt{\frac{\pi}{2}} \frac{\Gamma(\frac{2+sp}{2})}{\Gamma(\frac{3+sp}{2})} \rho^{-\frac{1+sp}{2}} K_{\frac{3+sp}{2}}(a\rho)
&\leq \int_{\rho}^{\infty} \frac{r^{-\frac{1+sp}{2}} K_{\frac{3+sp}{2}}(ar)}{\sqrt{\cosh r - \cosh \rho}} \, \mathrm{d}r \\
&\leq (1+\varepsilon) \sqrt{\frac{\pi}{2}} \frac{\Gamma(\frac{2+sp}{2})}{\Gamma(\frac{3+sp}{2})} \rho^{-\frac{1+sp}{2}} K_{\frac{3+sp}{2}}(a\rho)
\end{split}
\end{equation}
for all $\rho \in (0, \tilde{\delta}_{s})$. Moreover, we may assume that $\tilde{\delta}_{s}$ has been chosen continuously on $s$ since $\{K_\nu\}_{\nu\in [3/2,(3+p)/2]}$ is equicontinuous. Let $\tilde{\delta} = \delta \land \min_{s \in [0,1]} \tilde{\delta}_{s}$, then $\delta = \delta(\varepsilon, p, R, a) > 0$ and \eqref{eq:delta-intK} holds for all $\rho \in (0, \tilde{\delta})$.

We denote by $H_{s, p, m, a}(\rho)$ the integrand in the left-hand side of \eqref{eq:zero-even}. Then, the same argument as in the odd-dimensional case shows
\begin{equation*}
\lim_{s \nearrow 1} (1-s) \int_{0}^{R} H_{s, p, m, a}(\rho) \, \mathrm{d}\rho = \lim_{s \nearrow 1} (1-s) \int_{0}^{\frac{\delta}{2a}} H_{s, p, m, a}(\rho) \, \mathrm{d}\rho.
\end{equation*}
We argue by induction again to prove \eqref{eq:zero-even}. If $m=1$, then
\begin{equation*}
H_{s, p, 1, a}(\rho) = a \rho^{p} \sinh\rho \int_{\rho}^{\infty} \frac{1}{\sqrt{\cosh r - \cosh \rho}} r^{-\frac{1+sp}{2}} K_{\frac{3+sp}{2}}(ar) \, \mathrm{d}r.
\end{equation*}
Since $a \geq 1/2$, we have $\rho < \delta$ and $a\rho < \delta$ for $\rho < \frac{\delta}{2a}$. Thus, by \eqref{eq:delta-sinh}, \eqref{eq:delta-intK}, and \eqref{eq:delta-K}, we obtain
\begin{equation*}
\begin{split}
(1-\varepsilon)^{3} \sqrt{\frac{\pi}{2}} \left( \frac{2}{a} \right)^{\frac{1+sp}{2}} \Gamma\left( \frac{2+sp}{2} \right) \rho^{p(1-s)-1}
&\leq H_{s, p, 1, a}(\rho) \\
&\leq (1+\varepsilon)^{3} \sqrt{\frac{\pi}{2}} \left( \frac{2}{a} \right)^{\frac{1+sp}{2}} \Gamma\left( \frac{2+sp}{2} \right) \rho^{p(1-s)-1}.
\end{split}
\end{equation*}
Therefore, we have
\begin{equation*}
\begin{split}
(1-\varepsilon)^{3} \sqrt{\frac{\pi}{2}} \frac{1}{p} \left( \frac{2}{a} \right)^{\frac{p+1}{2}} \Gamma\left( \frac{p+3}{2} \right)
&\leq \lim_{s \nearrow 1} (1-s) \int_{0}^{R} H_{s, p, 1, a}(\rho) \, \mathrm{d}\rho \\
&\leq (1+\varepsilon)^{3} \sqrt{\frac{\pi}{2}} \frac{1}{p} \left( \frac{2}{a} \right)^{\frac{p+1}{2}} \Gamma\left( \frac{p+3}{2} \right),
\end{split}
\end{equation*}
from which we deduce \eqref{eq:zero-even} for $m=1$ by taking $\varepsilon \to 0$.

Suppose that \eqref{eq:zero-even} is true for $m \geq 1$. Then, by \eqref{eq:delta-sinh}, \Cref{lem:recurrence}, and \Cref{lem:positivity}, we have
\begin{equation*}
\begin{split}
&\lim_{s \nearrow 1} (1-s) \int_{0}^{R} H_{s, p, m+1, a}(\rho) \, \mathrm{d}\rho \\
&= \lim_{s \nearrow 1} (1-s) \int_{0}^{\frac{\delta}{2a}} \rho^{p} \sinh^{2m+1}\rho \int_{\rho}^{\infty} \frac{\sinh r}{\sqrt{\cosh r - \cosh \rho}} \left( \frac{-\partial_{r}}{\sinh r} \right)^{m+1} \widetilde{K}_{\frac{1+sp}{2},a}(r) \, \mathrm{d}r \, \mathrm{d}\rho \\
&\leq \lim_{s \nearrow 1} (1-s) (1+\varepsilon)^{2m} \int_{0}^{\frac{\delta}{2a}} \rho^{p+2m} (-\partial_{\rho}) \int_{\rho}^{\infty} \frac{\sinh r}{\sqrt{\cosh r - \cosh \rho}} \left( \frac{-\partial_{r}}{\sinh r} \right)^{m} \widetilde{K}_{\frac{1+sp}{2},a}(r) \, \mathrm{d}r \, \mathrm{d}\rho.
\end{split}
\end{equation*}
Using the integration by parts, \eqref{eq:delta-sinh}, and the induction hypothesis, we arrive at
\begin{equation*}
\begin{split}
&\lim_{s \nearrow 1} (1-s) \int_{0}^{R} H_{s, p, m+1, a}(\rho) \, \mathrm{d}\rho \\
&= (1+\varepsilon)^{2m} (p+2m) \\
&\quad\times \lim_{s \nearrow 1} (1-s) \int_{0}^{\frac{\delta}{2a}} \rho^{p+2m-1} \int_{\rho}^{\infty} \frac{\sinh r}{\sqrt{\cosh r - \cosh \rho}} \left( \frac{-\partial_{r}}{\sinh r} \right)^{m} \widetilde{K}_{\frac{1+sp}{2},a}(r) \, \mathrm{d}r \, \mathrm{d}\rho \\
&\leq \frac{(1+\varepsilon)^{2m}}{(1-\varepsilon)^{2m-1}} (p+2m) \lim_{s \nearrow 1} (1-s) \int_{0}^{\frac{\delta}{2a}} H_{s, p, m, a}(\rho) \, \mathrm{d}\rho \\
&= \frac{(1+\varepsilon)^{2m}}{(1-\varepsilon)^{2m-1}} \sqrt{\frac{\pi}{2}} \frac{2^{m}}{p} \left( \frac{2}{a} \right)^{\frac{p+1}{2}} \Gamma\left(\frac{p+2m+2}{2} \right).
\end{split}
\end{equation*}
The inequality
\begin{equation*}
\lim_{s \nearrow 1} (1-s) \int_{0}^{R} H_{s, p, m+1, a}(\rho) \, \mathrm{d}\rho \geq \frac{(1-\varepsilon)^{2m}}{(1+\varepsilon)^{2m-1}} \sqrt{\frac{\pi}{2}} \frac{2^{m}}{p} \left( \frac{2}{a} \right)^{\frac{p+1}{2}} \Gamma\left( \frac{p+2m+2}{2} \right)
\end{equation*}
can be obtained in the same way. Thus, we conclude that \eqref{eq:zero-even} for $m+1$ holds by taking $\varepsilon \to 0$. This finishes the proof for the even-dimensional case.
\end{proof}

\begin{lemma} \label{lem:zero-beta}
Let $n\geq 2$ and $p > 1$. For any $R > 0$ and $\beta > 0$,
\begin{equation} \label{eq:zero-beta}
\lim_{s \nearrow 1} c_{n,s,p} \int_{0}^{R} \rho^{p+\beta} \mathcal{K}_{n,s,p}(\rho)\sinh^{n-1}\rho \,\mathrm{d}\rho = 0.
\end{equation}
\end{lemma}

\begin{proof}
We proceed as in the previous lemma to prove \eqref{eq:zero-beta}. When $n=2m+1$ with $m \geq 1$, we show
\begin{equation*}
\lim_{s \nearrow 1} (1-s) \int_{0}^{R} \rho^{p+\beta} \sinh^{2m}\rho \left( \frac{-\partial_\rho}{\sinh \rho} \right)^{m} \widetilde{K}_{\frac{1+sp}{2},a}(\rho) \,\mathrm{d}\rho = 0 \quad\text{for each } a \geq 1
\end{equation*}
by induction. Indeed, for $\varepsilon \in (0,1)$ let $\delta > 0$ be the constant given in the proof of \Cref{lem:zero}. Then, by using \eqref{eq:delta-sinh} and \eqref{eq:delta-K} we prove
\begin{equation*}
\begin{split}
&\lim_{s \nearrow 1} (1-s) \int_{0}^{R} \rho^{\beta} G_{s, p, 1, a}(\rho) \,\mathrm{d}\rho \\
&\leq \lim_{s \nearrow} (1-s)(1+\varepsilon)^{2} \left( \frac{2}{a} \right)^{\frac{1+sp}{2}} \Gamma\left( \frac{3+sp}{2} \right) \int_{0}^{\delta/a} \rho^{p(1-s)+\beta-1} \,\mathrm{d}\rho \\
&= \lim_{s \nearrow} (1-s)(1+\varepsilon)^{2} \left( \frac{2}{a} \right)^{\frac{1+sp}{2}} \Gamma\left( \frac{3+sp}{2} \right) \frac{1}{p(1-s)+\beta} \left( \frac{\delta}{a} \right)^{p(1-s)+\beta} = 0
\end{split}
\end{equation*}
for the case $m=1$, where $G_{s, p, m, a}$ is the function defined in the proof of \Cref{lem:zero}. Moreover, one can follow the steps in the proof of \Cref{lem:zero} to obtain
\begin{equation*}
\begin{split}
&\frac{(1-\varepsilon)^{2m+1}}{(1+\varepsilon)^{2m}} (p+\beta+2m+1) \lim_{s \nearrow 1} (1-s) \int_{0}^{R} \rho^{\beta} G_{s, p, m, a}(\rho) \,\mathrm{d}\rho \\
&\leq \lim_{s \nearrow 1} (1-s) \int_{0}^{R} \rho^{\beta} G_{s, p, m+1, a}(\rho) \,\mathrm{d}\rho \\
&\leq \frac{(1+\varepsilon)^{2m+1}}{(1-\varepsilon)^{2m}} (p+\beta+2m+1) \lim_{s \nearrow 1} (1-s) \int_{0}^{R} \rho^{\beta} G_{s, p, m, a}(\rho) \,\mathrm{d}\rho,
\end{split}
\end{equation*}
which proves the induction step.

The even-dimensional case $n=2m$ with $m \geq 1$ can also be verified by proving
\begin{equation*}
\begin{split}
\lim_{s \nearrow 1} (1-s) \int_{0}^{R} \rho^{p+\beta} \sinh^{2m-1}\rho \int_{\rho}^{\infty} \frac{\sinh r}{\sqrt{\cosh r - \cosh \rho}} \left( \frac{-\partial_{r}}{\sinh r} \right)^{m} \widetilde{K}_{\frac{1+sp}{2},a}(r) \, \mathrm{d}r \,\mathrm{d}\rho = 0
\end{split}
\end{equation*}
for each $a \geq 1/2$. This can be proved by the induction as in the previous lemma, so we omit the proof.
\end{proof}

Let us provide the proof of \Cref{thm:convergence} by using the pointwise representation \eqref{eq:pw-representation} and Taylor's theorem, and gathering pieces of limits in the preceding lemmas.

\begin{proof} [Proof of \Cref{thm:convergence}]
Let $u \in C^{2}_{b}(\mathbb{H}^{n})$ and let $x \in \mathbb{H}^{n}$ be such that $\nabla u(x) \neq 0$. Let $R > 0$, then by \Cref{lem:infty} we first have
\begin{equation*}
\left| c_{n,s,p} \int_{d(x,\xi) \geq R} \Phi_{p}(u(x)-u(\xi)) \mathcal{K}_{n,s,p}(d(x,\xi)) \,\mathrm{d} \xi \right| \lesssim c_{n,s,p} \int_{R}^{\infty} \mathcal{K}_{n,s,p}(\rho) \sinh^{n-1}\rho \,\mathrm{d}\rho \to 0
\end{equation*}
as $s \to 1^{-}$. Thus, by the pointwise representation \eqref{eq:pw-representation} of the fractional $p$-Laplacian, we obtain
\begin{equation} \label{eq:limit-1}
\lim_{s \nearrow 1} (-\Delta_{\mathbb{H}^{n}})^{s}_{p} u(x) = \lim_{s \nearrow 1} c_{n,s,p} \mathrm{P.V.} \int_{d(x, \xi)<R} \Phi_{p}(u(x)-u(\xi)) \mathcal{K}_{n,s,p}(d(x, \xi)) \,\mathrm{d}\xi.
\end{equation}
Let $v = \exp_{x}^{-1}\xi$ be a tangent vector in $T_{x}\mathbb{H}^{n}$ and denote by $\mathcal{T}_{x}\xi$ the point $\exp_{x}(-v) \in \mathbb{H}^{n}$. Since $\mathcal{K}_{n,s,p}(d(x,\xi)) = \mathcal{K}_{n,s,p}(d(x,\mathcal{T}_{x}\xi))$, we write
\begin{equation*}
\begin{split}
&\int_{d(x, \xi)<R} \Phi_{p}(u(x)-u(\xi)) \mathcal{K}_{n,s,p}(d(x, \xi)) \,\mathrm{d}\xi \\
&= \frac{1}{2} \int_{d(x, \xi)<R} |u(x)-u(\xi)|^{p-2}(2u(x)-u(\xi)-u(\mathcal{T}_{x}\xi)) \mathcal{K}_{n,s,p}(d(x,\xi)) \,\mathrm{d}\xi \\
&\quad + \frac{1}{2} \int_{d(x, \xi) < R} \left( |u(x)-u(\mathcal{T}_{x}\xi)|^{p-2} - |u(x)-u(\xi)|^{p-2} \right)(u(x)-u(\mathcal{T}_{x}\xi)) \mathcal{K}_{n,s,p}(d(x,\xi)) \,\mathrm{d}\xi \\
&=: J_{1} + J_{2}.
\end{split}
\end{equation*}
By Taylor's theorem, we have
\begin{equation*}
u(x)-u(\xi) = -\langle \nabla u(x), v \rangle + O(|v|^{2}), \quad u(x)-u(\mathcal{T}_{x}\xi) = \langle \nabla u(x), v\rangle + O(|v|^{2}),
\end{equation*}
and
\begin{equation*}
2u(x)-u(\xi)-u(\mathcal{T}_{x}\xi) = -\langle D^{2}u(x) v, v\rangle + O(|v|^{3}).
\end{equation*}
If we write $\omega = v/|v|$, then
\begin{equation*}
|u(x) - u(\xi)|^{p-2} = |v|^{p-2} |\langle \nabla u(x), \omega \rangle|^{p-2} + O(|v|^{p-1}).
\end{equation*}
Thus, we obtain
\begin{equation*}
|u(x)-u(\xi)|^{p-2}(2u(x)-u(\xi)-u(\mathcal{T}_{x}\xi)) = -|v|^{p} |\langle \nabla u(x), \omega \rangle|^{p-2} \langle D^{2}u(x) \omega, \omega \rangle + O(|v|^{p+1}).
\end{equation*}
Therefore, we deduce
\begin{equation} \label{eq:J-1}
\begin{split}
J_{1}
&= -\frac{1}{2} \int_{0}^{R} \int_{\mathbb{S}^{n-1}} \rho^{p} |\langle \nabla u(x), \omega \rangle|^{p-2} \langle D^{2}u(x) \omega, \omega \rangle \mathcal{K}_{n,s,p}(\rho) \sinh^{n-1} \rho \, \mathrm{d}\omega \,\mathrm{d}\rho \\
&\quad + \frac{1}{2} \int_{d(x, \xi)<R} O(d(x, \xi)^{p+1}) \mathcal{K}_{n,s,p}(d(x, \xi)) \,\mathrm{d}\xi.
\end{split}
\end{equation}
For $J_{2}$, since
\begin{equation*}
\begin{split}
&|u(\mathcal{T}_{x}\xi)-u(x)|^{p-2} - |u(x)-u(\xi)|^{p-2} \\
&= (p-2)|v|^{p-1} \langle \nabla u(x), \omega \rangle |\langle \nabla u(x), \omega \rangle|^{p-4} \langle D^{2}u(x) \omega, \omega \rangle + O(|v|^{p}),
\end{split}
\end{equation*}
we have
\begin{equation*}
\begin{split}
&\left( |u(\mathcal{T}_{x}\xi)-u(x)|^{p-2} - |u(x)-u(\xi)|^{p-2} \right)(u(x)-u(\mathcal{T}_{x}\xi)) \\
&= -(p-2)|v|^{p} |\langle \nabla u(x), \omega \rangle|^{p-2} \langle D^{2}u(x) \omega, \omega \rangle + O(|v|^{p+1}).
\end{split}
\end{equation*}
Thus, we obtain
\begin{equation} \label{eq:J-2}
\begin{split}
J_{2} 
&= -\frac{p-2}{2} \int_{0}^{R} \int_{\mathbb{S}^{n-1}} \rho^{p} |\langle \nabla u(x), \omega \rangle|^{p-2} \langle D^{2}u(x) \omega, \omega \rangle \mathcal{K}_{n,s,p}(\rho) \sinh^{n-1} \rho \, \mathrm{d}\omega \,\mathrm{d}\rho \\
&\quad + \frac{1}{2} \int_{d(x, \xi)<R} O(d(x, \xi)^{p+1}) \mathcal{K}_{n,s,p}(d(x, \xi)) \,\mathrm{d}\xi.
\end{split}
\end{equation}
Combining \eqref{eq:limit-1}, \eqref{eq:J-1}, and \eqref{eq:J-2}, and using \Cref{lem:zero} and \Cref{lem:zero-beta}, we arrive at
\begin{equation*}
\lim_{s\nearrow 1} (-\Delta_{\mathbb{H}^{n}})^{s}_{p} u(x) = -\frac{p-1}{2} \frac{1}{\pi^{\frac{n-1}{2}}} \frac{\Gamma(\frac{p+n}{2})}{\Gamma(\frac{p+1}{2})} \int_{\mathbb{S}^{n-1}} |\langle \nabla u(x), \omega \rangle|^{p-2} \langle D^{2}u(x) \omega, \omega \rangle \,\mathrm{d}\omega.
\end{equation*}
The argument as in the proof of \cite[Theorem 2.8]{BS22} shows
\begin{equation*}
\int_{\mathbb{S}^{n-1}} |\langle \nabla u(x), \omega \rangle|^{p-2} \langle D^{2}u(x) \omega, \omega \rangle \,\mathrm{d}\omega = \gamma_{p} (\Delta_{\mathbb{H}^{n}})_{p}
\end{equation*}
when $\nabla u(x) \neq 0$, where
\begin{equation} \label{eq:gamma-p}
\gamma_{p} = \int_{\mathbb{S}^{n-1}} |\omega_{n}|^{p-2} \omega_{1}^{2} \,\mathrm{d}\omega = \pi^{\frac{n-1}{2}} \frac{\Gamma(\frac{p-1}{2})}{\Gamma(\frac{p+n}{2})}.
\end{equation}
See \cite[Lemma 2.1]{IN10} for the computation of \eqref{eq:gamma-p}. This finishes the proof.
\end{proof}


\begin{appendix}


\section{Auxiliary result} \label{sec:appendix}


In this section, we recall an auxiliary result from \cite{dTGCV21} that helps proving \Cref{def:kernel} in \Cref{sec:kernel}.

\begin{lemma} \label{lem:appendix}
Let $p > 1$, $r > 0$, $u \in C^{2}_{b}(\mathbb{H}^{n})$, and $x \in \mathbb{H}^{n}$. If $p \in (1,\frac{2}{2-s}]$, assume $\nabla u(x) \neq 0$ additionally. If $K : \mathbb{H}^{n} \to \mathbb{R}$ is rotationally symmetric with respect to $x$, that is, $K(\xi)=K(d(x,\xi))$ for all $\xi \in \mathbb{H}^{n}$, and $\int_{\mathbb{H}^n} K(\xi)|\xi|^\alpha \,d\xi < \infty$, then
\begin{equation*}
\left| \mathrm{P.V.} \int_{d(x, \xi)<r} \Phi_{p}(u(x)-u(\xi)) K(d(x,\xi)) \,\mathrm{d}\xi \right| \leq C \int_{d(x, \xi)<r} d(x,\xi)^{\alpha} |K(d(x,\xi))| \,\mathrm{d}\xi
\end{equation*}
for some constant $C = C(n, p, \|u\|_{C^{2}(\mathbb{H}^{n})}) > 0$, where $\alpha = 2p-2$ when $p \in (\frac{2}{2-s}, 2)$ and $\alpha = p$ otherwise.
\end{lemma}

The cases $p \in [2, \infty)$, $p \in (1, \frac{2}{2-s}]$, and $p \in (\frac{2}{2-s}, 2)$                                                                                                                                                                        are proved in \cite[Lemma A.1, A2, and A3]{dTGCV21}, respectively, for the case of Euclidean spaces. We omit the proof of \Cref{lem:appendix} because the same proofs work in our framework.

\end{appendix}



\begin{thebibliography}{10}

\bibitem{AOCM18}
D.~Alonso-Or\'{a}n, A.~C\'{o}rdoba, and A.~D. Mart\'{\i}nez.
\newblock Integral representation for fractional {L}aplace-{B}eltrami
  operators.
\newblock {\em Adv. Math.}, 328:436--445, 2018.

\bibitem{BGS15}
V.~Banica, M.~d.~M. Gonz\'{a}lez, and M.~S\'{a}ez.
\newblock Some constructions for the fractional {L}aplacian on noncompact
  manifolds.
\newblock {\em Rev. Mat. Iberoam.}, 31(2):681--712, 2015.

\bibitem{Boc49}
S.~Bochner.
\newblock Diffusion equation and stochastic processes.
\newblock {\em Proc. Nat. Acad. Sci. U.S.A.}, 35:368--370, 1949.

\bibitem{BS22}
C.~Bucur and M.~Squassina.
\newblock An asymptotic expansion for the fractional {$p$}-{L}aplacian and for
  gradient-dependent nonlocal operators.
\newblock {\em Commun. Contemp. Math.}, 24(4):Paper No. 2150021, 34, 2022.

\bibitem{CS07}
L.~Caffarelli and L.~Silvestre.
\newblock An extension problem related to the fractional {L}aplacian.
\newblock {\em Comm. Partial Differential Equations}, 32(7-9):1245--1260, 2007.

\bibitem{CS09}
L.~Caffarelli and L.~Silvestre.
\newblock Regularity theory for fully nonlinear integro-differential equations.
\newblock {\em Comm. Pure Appl. Math.}, 62(5):597--638, 2009.

\bibitem{CS17}
L.~A. Caffarelli and Y.~Sire.
\newblock On some pointwise inequalities involving nonlocal operators.
\newblock In {\em Harmonic analysis, partial differential equations and
  applications}, Appl. Numer. Harmon. Anal., pages 1--18.
  Birkh\"{a}user/Springer, Cham, 2017.

\bibitem{CKK+22}
L.~Capogna, J.~Kline, R.~Korte, N.~Shanmugalingam, and M.~Snipes.
\newblock Neumann problems for $p$-harmonic functions, and induced nonlocal
  operators in metric measure spaces.
\newblock {\em Amer. J. Math.}, 147(6):1653--1711, 2025.

\bibitem{CG11}
S.-Y.~A. Chang and M.~d.~M. Gonz\'{a}lez.
\newblock Fractional {L}aplacian in conformal geometry.
\newblock {\em Adv. Math.}, 226(2):1410--1432, 2011.

\bibitem{CK08}
Z.-Q. Chen and T.~Kumagai.
\newblock Heat kernel estimates for jump processes of mixed types on metric
  measure spaces.
\newblock {\em Probab. Theory Related Fields}, 140(1-2):277--317, 2008.

\bibitem{CRS+18}
O.~Ciaurri, L.~Roncal, P.~R. Stinga, J.~L. Torrea, and J.~L. Varona.
\newblock Nonlocal discrete diffusion equations and the fractional discrete
  {L}aplacian, regularity and applications.
\newblock {\em Adv. Math.}, 330:688--738, 2018.

\bibitem{CRT01}
T.~Coulhon, E.~Russ, and V.~Tardivel-Nachef.
\newblock Sobolev algebras on {L}ie groups and {R}iemannian manifolds.
\newblock {\em Amer. J. Math.}, 123(2):283--342, 2001.

\bibitem{DNS19}
P.~L. De~N\'{a}poli and P.~R. Stinga.
\newblock Fractional {L}aplacians on the sphere, the {M}inakshisundaram zeta
  function and semigroups.
\newblock In {\em New developments in the analysis of nonlocal operators},
  volume 723 of {\em Contemp. Math.}, pages 167--189. Amer. Math. Soc.,
  [Providence], RI, [2019] \copyright 2019.

\bibitem{dTGCV21}
F.~del Teso, D.~G\'{o}mez-Castro, and J.~L. V\'{a}zquez.
\newblock Three representations of the fractional {$p$}-{L}aplacian:
  {S}emigroup, extension and {B}alakrishnan formulas.
\newblock {\em Fract. Calc. Appl. Anal.}, 24(4):966--1002, 2021.

\bibitem{DNPV12}
E.~Di~Nezza, G.~Palatucci, and E.~Valdinoci.
\newblock Hitchhiker's guide to the fractional {S}obolev spaces.
\newblock {\em Bull. Sci. Math.}, 136(5):521--573, 2012.

\bibitem{FF15}
F.~Ferrari and B.~Franchi.
\newblock Harnack inequality for fractional sub-{L}aplacians in {C}arnot
  groups.
\newblock {\em Math. Z.}, 279(1-2):435--458, 2015.

\bibitem{FMP+18}
F.~Ferrari, M.~Miranda, Jr., D.~Pallara, A.~Pinamonti, and Y.~Sire.
\newblock Fractional {L}aplacians, perimeters and heat semigroups in {C}arnot
  groups.
\newblock {\em Discrete Contin. Dyn. Syst. Ser. S}, 11(3):477--491, 2018.

\bibitem{FLW14}
R.~L. Frank, D.~Lenz, and D.~Wingert.
\newblock Intrinsic metrics for non-local symmetric {D}irichlet forms and
  applications to spectral theory.
\newblock {\em J. Funct. Anal.}, 266(8):4765--4808, 2014.

\bibitem{Gon18}
M.~d.~M. Gonz\'{a}lez.
\newblock Recent progress on the fractional {L}aplacian in conformal geometry.
\newblock In {\em Recent developments in nonlocal theory}, pages 236--273. De
  Gruyter, Berlin, 2018.

\bibitem{GZ03}
C.~R. Graham and M.~Zworski.
\newblock Scattering matrix in conformal geometry.
\newblock {\em Invent. Math.}, 152(1):89--118, 2003.

\bibitem{Gri03}
A.~Grigor'yan.
\newblock Heat kernels and function theory on metric measure spaces.
\newblock In {\em Heat kernels and analysis on manifolds, graphs, and metric
  spaces ({P}aris, 2002)}, volume 338 of {\em Contemp. Math.}, pages 143--172.
  Amer. Math. Soc., Providence, RI, 2003.

\bibitem{GHL14}
A.~Grigor'yan, J.~Hu, and K.-S. Lau.
\newblock Estimates of heat kernels for non-local regular {D}irichlet forms.
\newblock {\em Trans. Amer. Math. Soc.}, 366(12):6397--6441, 2014.

\bibitem{GN98}
A.~Grigor'yan and M.~Noguchi.
\newblock The heat kernel on hyperbolic space.
\newblock {\em Bull. London Math. Soc.}, 30(6):643--650, 1998.

\bibitem{GZZ18}
L.~Guo, B.~Zhang, and Y.~Zhang.
\newblock Fractional {$p$}-{L}aplacian equations on {R}iemannian manifolds.
\newblock {\em Electron. J. Differential Equations}, pages Paper No. 156, 17,
  2018.

\bibitem{IN10}
H.~Ishii and G.~Nakamura.
\newblock A class of integral equations and approximation of {$p$}-{L}aplace
  equations.
\newblock {\em Calc. Var. Partial Differential Equations}, 37(3-4):485--522,
  2010.

\bibitem{JB96}
C.~M. Joshi and S.~K. Bissu.
\newblock Inequalities for some special functions.
\newblock {\em J. Comput. Appl. Math.}, 69(2):251--259, 1996.

\bibitem{KKL21}
J.~Kim, M.~Kim, and K.-A. Lee.
\newblock Harnack inequality for fractional {L}aplacian-type operators on
  hyperbolic spaces.
\newblock {\em arXiv preprint arXiv:2108.08746}, 2021.

\bibitem{KKL19}
J.~Korvenp\"a\"a, T.~Kuusi, and E.~Lindgren.
\newblock Equivalence of solutions to fractional {$p$}-{L}aplace type
  equations.
\newblock {\em J. Math. Pures Appl. (9)}, 132:1--26, 2019.

\bibitem{Kwa17}
M.~Kwa\'{s}nicki.
\newblock Ten equivalent definitions of the fractional {L}aplace operator.
\newblock {\em Fract. Calc. Appl. Anal.}, 20(1):7--51, 2017.

\bibitem{OLBC10}
F.~W.~J. Olver, D.~W. Lozier, R.~F. Boisvert, and C.~W. Clark, editors.
\newblock {\em N{IST} handbook of mathematical functions}.
\newblock U.S. Department of Commerce, National Institute of Standards and
  Technology, Washington, DC; Cambridge University Press, Cambridge, 2010.
\newblock With 1 CD-ROM (Windows, Macintosh and UNIX).

\bibitem{ST10}
P.~R. Stinga and J.~L. Torrea.
\newblock Extension problem and {H}arnack's inequality for some fractional
  operators.
\newblock {\em Comm. Partial Differential Equations}, 35(11):2092--2122, 2010.

\end{thebibliography}

\end{document}